\documentclass[11pt]{article}

\usepackage[font=small,labelfont=bf]{caption}

\usepackage{enumitem}
\usepackage{graphicx}
\usepackage{xcolor}
\usepackage{blkarray}
\usepackage{amsmath,amssymb}
\usepackage{bbm}
\usepackage{mathtools}
\usepackage{pdflscape}
\usepackage{float}
\usepackage{rotating}
\usepackage{titlesec}
\usepackage{subcaption}
\usepackage{longtable}

\usepackage[most]{tcolorbox}
\tcbuselibrary{skins}

\usepackage{amsthm}

\newtheorem{theorem}{Theorem}[section]
\newtheorem{corollary}{Corollary}[theorem]

\newtheorem{lemma}{Lemma}[theorem]

\newtheorem{remark}{Remark}[section]
\newtheorem{example}{Example}[section]

\newenvironment{subproof}[1][\proofname]{%
  \begin{proof}[#1]%
}{%
  \end{proof}%
}

\theoremstyle{break}
\newtheorem{assumption}{Assumption}

\newtheorem{assumptionalt}{Assumption}
\newenvironment{assumptionp}[1]{
  \renewcommand\theassumptionalt{#1}
  \assumptionalt
}{\endassumptionalt}

\newtcolorbox[auto counter,list inside=examplebox]{examplebox}[2][]{%
    floatplacement=h, 
    float,
    title={Box~\thetcbcounter:~#2},
    #1,
    skin=enhanced,
    titlerule=0.5pt,
    arc=0pt,
    outer arc=0pt,
    sharp corners,
}

\usepackage[
citestyle=numeric,
style=numeric,
uniquename=false,
sorting=nyt,
maxbibnames=8
]{biblatex}

\allowdisplaybreaks

\titleformat{\paragraph}
{\normalfont\normalsize\bfseries}{\theparagraph}{1em}{}
\titlespacing{\paragraph}
{0pt}{3.25ex plus 1ex minus .2ex}{1.5ex plus .2ex}

\DeclareMathOperator{\EX}{\mathbbm{E}}

\begin{document}
\setlength{\abovedisplayskip}{4pt}
\setlength{\belowdisplayskip}{10pt}
\setlength{\abovedisplayshortskip}{4pt}
\setlength{\belowdisplayshortskip}{10pt}

\title{A multitype Markovian branching process with one-type population size dependence}

\author{Somya Mehra$^{1,2,3}$, Peter G Taylor$^{3}$}

\date{}

\maketitle

\noindent 
1: Nuffield Department of Medicine, University of Oxford, Oxford, United Kingdom \\
2: Mahidol Oxford Tropical Medicine Research Unit, Mahidol University, Bangkok, Thailand \\
3: School of Mathematics and Statistics, University of Melbourne, Parkville, Australia

\begin{abstract}
Motivated by a within-host framework of immunity-modulated parasitic disease, we formulate a multitype Markovian branching process with reproductive parameters dependent on the number of individuals of a single type only. We impose a soft carrying capacity $K$ with respect to the controlling type, serving as a threshold between subcritical and supercritical dynamics. We prove that extinction occurs almost surely when the offspring batch of each non-controlling type includes an individual of the controlling type with positive probability, with the time to extinction possessing finite moments of all orders. We then construct a sequence of density-dependent processes indexed by $K$ and study scaling limits of the process using the canonical functional law of large numbers (FLLN) and central limit theorem, with a view towards characterising the emergence of metastable behaviour in the vicinity of asymptotically stable equilibria, or stable limit cycles of the FLLN.
\end{abstract}

\section{Introduction}

We consider a state-dependent multitype Markovian branching process with a carrying capacity $K$ imposed on a \textit{single} type in the system: when the size of the subpopulation of this type exceeds (falls below) $K$, the reproductive parameters of the system would give rise to a subcritical (supercritical) Markovian branching process.

Our motivating example is a within-host model of a parasitic disease, where the controlling `type' can be interpreted as an abstracted immunity functional akin to that suggested in \parencite{mehra2024hybrid} which governs the dynamics of (potentially multi-stage) parasite clearance and replication. We can formulate this as a $d$-type branching process, where types $2, \dots, d$ represent different parasite stages; type $1$ represents immunity `increments' that are gained with parasite exposure but are eventually lost (to model the waning of immunity in the absence of exposure) without generating further offspring \parencite{mehra2024hybrid}; and the reproductive parameters depend exclusively on the immunity level, or equivalently, the type $1$ subpopulation size. In particular, we seek to characterise the conditions under which this model exhibits quasi-equilibrium behaviour, as in chronic Chagas disease which is characterised by an apparent equilibrium of \textit{Trypanosoma cruzi} parasites in the bloodstream \parencite{de2022human} that can persist for decades without symptoms. We introduce this abstracted immunity functional with a view towards constructing a unified framework of acute and chronic infection, allowing us to distinguish acute infection dynamics (when an immune-naive individual is inoculated with a small number of parasites) from post-treatment dynamics for chronic infection (when a small number of parasites survive drug treatment in a partially-immune individual). This framework is tailored to a pathogen for which immune modulation, rather than host cell depletion, serves as the key controlling mechanism. 

In the single-type case, size-dependent branching process models with a soft carrying capacity have been analysed extensively by \parencite{jagers2011population, hamza2016establishment, hognas2019lifetime, jagers2020populations} and others, with some authors additionally accommodating age dependence. Qualitative dynamics for these single-type models are well-established: extinction is certain (under mild conditions) \parencite{jagers2020populations} but if certain subcriticality conditions are satisfied and reproduction decreases monotonically with the population size, then given an initial condition of order $o(K)$, there is a non-negligible probability that the process reaches the vicinity of the carrying capacity $K$ in order $\mathcal{O}(\log(K))$ time, where it then lingers for an extended time scale of $\mathcal{O}(\exp(cK))$ \parencite{jagers2011population, hamza2016establishment}. In the multitype case, the authors of \parencite{jagers2016size, fan2020limit} construct a general type, age and population size dependent framework, for which scaling limits are derived; however, the `carrying capacity' is not assigned a specific interpretation, and is instead treated as a natural scaling parameter. Results on extinction and asymptotic behaviour are available for nearly (super/sub)-critical multitype branching processes for which reproductive parameters converge in the infinite population size limit \parencite{klebaner1989linear, klebaner1994asymptotic, gonzalez2005unlimited}, but these conditions are not directly satisfied by a model in which reproductive parameters are a function of the number of individuals of a single type only, as we consider here.

The present paper is structured as follows. In Section \ref{sec::model_structure}, we make precise our notion of a $d$-type Markovian branching process with a soft carrying capacity imposed on the type 1 subpopulation. We assume that the offspring batch of each individual of types $2, \dots, d$ includes a type 1 individual with positive probability, modelling the idea that infections of any type increase immunity. Under this assumption, we prove that the system undergoes extinction almost surely, and that the time to extinction possesses finite moments of all orders (Theorem \ref{theorem::extinction_almost_surely}, Section \ref{sec::extinction_almost_surely}).

In Section \ref{sec::scaling_limits}, we construct a sequence of density-dependent processes indexed by the carrying capacity $K$, and analyse the Gaussian approximation that emerges in the limit $K \to \infty$ using the functional law of large numbers (FLLN) and central limit theorem (CLT) of \textcite{ethier2009markov} (Theorems \ref{theorem::FLLN} and \ref{theorem::central_limit}, based on Theorems 11.2.1 and 11.2.3 of \parencite{ethier2009markov}). Under an irreducibility assumption, we show that the system of ODEs constituting the FLLN possesses a unique non-zero equilibrium (Theorem \ref{theorem::FLLN_equilibrium}). In the planar case $d=2$, we recover a characterisation of the limiting behaviour of this system of ODEs based on the classification of the non-zero equilibrium (Theorem \ref{theorem::planar_FLLN_limiting_behaviour}). Under a monotonicity condition in the case $d=2$, encompassing the special case of fixed death rates and monotonically decaying mean offspring batch sizes, we prove that the unique non-zero equilibrium is asymptotically stable with domain of attraction $\mathcal{D} \supseteq \mathbbm{R}_{\geq 0} \times \mathbbm{R}_{> 0}$ (Corollary \ref{corollary::FLLN_limiting_behaviour_simple}). However, we show by counter-example that in the case $d \geq 3$, this monotonicity condition is neither sufficient to ensure the asymptotic stability of the (unique) non-zero equilibrium, nor rule out periodic orbits (Example \ref{example::3D_limit_cycle}).

The possible emergence of metastable behaviour is addressed in Section \ref{sec::metastable_behaviour}. Provided the non-zero equilibrium of the FLLN is asymptotically stable and the scaled system has reached its basin of attraction, we use the work of \textcite{prodhomme2023strong} to show that, for every $\Delta: \mathbbm{R}_{> 0} \to \mathbbm{R}_{> 0}$ such that $K^{-\frac{1}{2}} \leq \Delta(K) \ll 1$, the Gaussian approximation derived under the FLLN and CLT retains asymptotic precision $\Delta(K)$ on sub-exponential time scales of order $\mathcal{O}(\exp(\beta K \Delta(K))$ for some constant $\beta > 0$ thereafter (Theorem \ref{theorem::moderate_deviations}, based on Corollary 1.2. of \parencite{prodhomme2023strong}). Alternatively, provided the FLLN possesses a stable limit cycle and the scaled system hits a sufficiently small neighbourhood of this cycle, we use the work of \textcite{bressloff2020phase} to show that the scaled process remains in a neighbourhood of order $\mathcal{O}(\Delta(K))$ of the cycle for a sub-exponential time scale of order $\mathcal{O}(\exp(\beta^* K \Delta^2(K))$ thereafter, although divergence in phase may occur on a much earlier time scale (Theorem \ref{theorem::limit_cycle_metastability}, based on \parencite{bressloff2020phase}). Ascent to the type 1 carrying capacity is characterised in Section \ref{sec::ascent_carrying_capacity}. When offspring batch sizes decrease monotonically with the type 1 subpopulation size in the sense of strong first order stochastic dominance \parencite{kopa2018strong} and the process is initialised with at least one individual of types $2, \dots, d$, we show that the type 1 subpopulation reaches the threshold $a K$, $0 < a < 1$ with non-negligible probability within $\mathcal{O}(\log(K))$ time (Theorem \ref{theorem::logarithmic_ascent}).

We return to our motivating example of immunity-modulated chronic parasitic infection in Section \ref{sec::example_parasite_model}. Focussing on the two-type case, we show that the qualitative dynamics of the system are dependent on how the rate of parasite clearance $g(x)$ is modulated as a function of the scaled immunity level $x$. When $g(x)/x$ is a strictly decreasing function of $x$ and parasite batch sizes decay monotonically as a function of immunity, we recover a concrete characterisation of quasi-equilibrium behaviour in direct analogy to the results of \parencite{jagers2011population}, which are tailored to the single-type case. However, we also characterise the possible emergence of self-sustained oscillations in the parasite burden, including subcritical Hopf bifurcations (Example \ref{example::2D_limit_cycle}). We conclude with some remarks in Section \ref{sec::conclusion}.

\section{Model structure} \label{sec::model_structure}

We consider a $d$-type branching process where the reproductive parameters, and consequently the criticality of the system, are dependent only on the size of the subpopulation of type $1$. This construction falls under the general framework of age, type and population size dependent branching processes for which limit theorems are derived in \parencite{jagers2016size, fan2020limit}. Here, we seek to establish results encompassing the qualitative behaviour of the system given a carrying capacity imposed on the type 1 subpopulation only; the `carrying capacity' in the multitype frameworks of \parencite{jagers2016size, fan2020limit} is treated as a natural scaling parameter, rather than an interpretable quantity governing the criticality of the system. Hereafter, we let $\lVert \mathbf{x} \rVert := \sum^d_{i=1} x_i$ and $\lVert \mathbf{x} \rVert_2 := \sqrt{\sum^d_{i=1} x_i^2}$. We denote by $\delta_{i,j}$ the Kronecker delta and $\mathbf{e}_i$ the $i^\text{th}$ unit vector in $(\mathbbm{R}_{\geq 0})^d$. We set $\mathbf{1} = (1, \dots, 1)$ and define $I_d$ to be the $d \times d$ identity matrix.

Formally, we construct a Markovian branching process $\{ \mathbf{x}(t), t \geq 0\}$ on the state space $(\mathbbm{Z}_{\geq 0})^d$ with support  $\mathcal{J} \subseteq (\mathbbm{Z}_{\geq 0})^d$ for offspring batch sizes, and transitions of the form
\begin{align}
    \mathbf{x} \to \mathbf{x} + \boldsymbol{\ell} - \mathbf{e}_i \text{ at rate } \gamma_i(x_1) \pi_{i, \boldsymbol{\ell}}(x_1) x_i \qquad i \in \{1, \dots, d \}, \, \boldsymbol{\ell} \in \mathcal{J}. \label{eq::transition_rates}
\end{align}

As a function of the type $1$ subpopulation size $x_1$, we thus interpret:
\begin{itemize}
    \item $\gamma_i(x_1)$ to be the rate of death for of an individual of type $i$, assumed to be bounded and strictly positive.
    \item $\pi_{(i, \boldsymbol{\ell})}(x_1)$ to be the probability that an individual of type $i$ will give rise to an offspring batch $\boldsymbol{\ell} \in \mathcal{J}$ upon death.
\end{itemize}

We further define:
\begin{itemize}
    \item $\mu_{i,j}(x_1) = \sum_{\boldsymbol{\ell} \in \mathcal{J}} \ell_j \pi_{(i, \boldsymbol{\ell})}(x_1)$ to be the mean number of type $j$ offspring produced upon the death of an individual of type $i$, assumed to be uniformly bounded.
    \item $\chi_{i,j_1, j_2}(x_1) = \sum_{\boldsymbol{\ell} \in \mathcal{J}} \ell_{j_1} \ell_{j_2} \pi_{(i, \boldsymbol{\ell})}(x_1)$ to be the entries of the second moment matrix for the offspring distribution of type $i$ individuals, likewise assumed to be uniformly bounded.
\end{itemize}

Set $M(x_1) = (\mu_{i,j}(x_1))_{1 \leq i,j \leq d}$ and $\Gamma(x_1) = \text{diag} \{ \gamma_1(x_1), \dots, \gamma_d(x_1)\}$.

By construction, this process is conservative and stable (Definition 4.3 of \parencite{kijima2013markov}). Uniform boundedness of the individual death rates $\gamma_i(x_1)$ and mean offspring batch sizes $\mu_{i,j}(x_1)$ ensure the existence of a constant $C>0$ such that, for all $\mathbf{x}$,
\begin{align}
    \sum^d_{i=1} \gamma_i(x_1) x_i \sum^d_{j=1} | \mu_{i,j}(x_1) - \delta_{i,j} | \leq C \lVert \mathbf{x} \rVert. \label{eq::regularity_condition}
\end{align}

By Theorem 2.1 of \textcite{meyn1993stability}, condition (\ref{eq::regularity_condition}) is sufficient to ensure both regularity, whereby the process remains finite almost surely on finite intervals, and integrability, with
\begin{align*}
    \EX[ \lVert \mathbf{x}(t) \rVert ] \leq e^{Ct} \lVert \mathbf{x}(0) \rVert \text{ for all } t.
\end{align*}

The reasoning of Theorem 1 of \textcite{hamza1995conditions} additionally implies that condition (\ref{eq::regularity_condition}) ensures non-explosivity, whereby the process almost surely undergoes finitely many jumps in finite time intervals. By Theorem 1.11 of \textcite{chen1991three}, condition (\ref{eq::regularity_condition}) also ensures that the transition rates (\ref{eq::transition_rates}) yield a unique $Q$-process.

Let $\lambda_\text{max}(x_1)$ be the spectral abscissa of the matrix $\big( \Gamma(x_1)( M(x_1) - I_d) \big)^T$. Then in the state $\mathbf{x} = (x_1, \dots, x_d)$, we say that the system is subcritical if $\lambda_\text{max}(x_1) < 0$, critical if $\lambda_\text{max}(x_1) = 0$ and supercritical if $\lambda_\text{max}(x_1) > 0$ \parencite{klebaner1991asymptotic}. Our central notion of a type 1 carrying capacity is defined below.

\begin{assumption} \label{assumption::type_1_carrying_capacity}
We assume that there is a number $K$ such that $\text{sign} (\lambda_\text{max}(x_1)) = \text{sign}(K - x_1)$; that is, the system is supercritical when $x_1 < K$, critical when $x_1=K$ and subcritical when $x_1 > K$. We call the number $K$ the carrying capacity for individuals of type $1$. We additionally assume that mean offspring batch sizes $\mu_{i,j}(x_1)$, second order offspring moments $\chi_{i,j_1, j_2}(x_1)$ and death rates $\gamma_i(x_1)$ are uniformly bounded (with $\inf_{x_1 \in \mathbbm{Z}_{\geq 0}, i \in \{1, \dots, d\}} \gamma_i(x_1) > 0$) and converge in the limits $x_1 \to 0$ and $x_1 \to \infty$ respectively.
\end{assumption}

To ensure that the type 1 subpopulation exerts adequate control on the branching process, so that the types $2, \dots, d$ subpopulations cannot grow indefinitely without accompanying growth of the type 1 subpopulation, we additionally enforce Assumption \ref{assumption::type_1_offspring} below.

\begin{assumption} \label{assumption::type_1_offspring}
    With probability at least $\varepsilon > 0$, the offspring batch associated with the death of an individual of type $2, \dots, d$ includes at least one type 1 individual, irrespective of the state of the system, that is, for all $x_1 \geq 0$ and $i \in \{2, \dots, d\}$, $\sum_{\boldsymbol{\ell} \in \mathcal{J}} \ell_1 \pi_{(i, \boldsymbol{\ell})}(x_1) \geq \varepsilon$.
\end{assumption}

For several results, we enforce the supplementary Assumption \ref{assumption::bounded_batch_size} that the offspring batch sizes have uniformly bounded support.

\begin{assumption} \label{assumption::bounded_batch_size}
There exists a constant $\Xi$ such that the support of the offspring batch size distribution for any individual of type $1, \dots, d$ lies in the finite set $[0, \Xi/d]^d$ for all $x_1 \geq 0$.
\end{assumption}

Due to the fact that the criticality of the process depends on the number of individuals of a single type only, the idea behind Assumptions \ref{assumption::type_1_carrying_capacity} and \ref{assumption::type_1_offspring} differs slightly from previous models in the literature. In the multitype case, \textcite{klebaner1994asymptotic} defines nearly (super/sub)-critical branching processes based on the limiting value of the reproductive parameters as $\lVert \mathbf{x} \rVert \to \infty$. Existing results pertaining to the extinction \parencite{gonzalez2005unlimited} and asymptotic behaviour \parencite{klebaner1994asymptotic} of multitype Markov population processes likewise assume parameter convergence in the limit $\lVert \mathbf{x} \rVert \to \infty$. Here, exclusive dependence on the type 1 subpopulation means that the reproductive parameters converge in the limit $x_1 \to \infty$ but not necessarily $\lVert \mathbf{x} \rVert \to \infty$. This means that the results of \parencite{klebaner1994asymptotic, gonzalez2005unlimited} do not generalise immediately to our setting. As related frameworks, we also mention a two-type age-dependent model of sexual reproduction formulated by \parencite{hamza2016establishment, fan2020limit} that considers a single reproducing type, but with reproductive parameters that can vary with the number of individuals of the non-reproducing type; and a two-type model of haemotopoesis formulated by \parencite{wang2025stochastic}, with reproductive parameters depending on the post-mitotic cell population only.

In the single-type case, \textcite{jagers2020populations} define a soft-carrying capacity as the threshold after which the expected step-change in the population size is strictly negative, and then prove (under relatively mild conditions) that extinction occurs almost surely. In the multitype case, however, the subcritical parameter regimes in place when the type 1 subpopulation exceeds the soft carrying capacity do not necessarily translate to a negative expected step change in the size of each subpopulation. The qualitative dynamics of a single type process with a soft carrying capacity are analysed in \textcite{jagers2011population} under a slightly stronger subcriticality condition, and the assumption that reproduction `decreases' monotonically with the population size: with non-zero probability, it is shown the single-type system approaches the vicinity of its carrying capacity $K$ in order $\mathcal{O}(\log(K))$ time, where it then lingers for an extended period of order $\mathcal{O}(\exp(cK))$. The proofs presented in \parencite{jagers2011population, jagers2020populations} generalise in part to the multitype case if the system is subcritical in a sufficiently strict sense except on a \textit{bounded} domain $\mathcal{D} \in (\mathbbm{Z}_{\geq 0})^d$; but the interaction between the respective carrying capacities for different subpopulations may be nuanced, and broad generalisations to the multitype case are not straightforward \parencite{jagers2016size}. In our case the system is supercritical on the unbounded domain $\mathcal{D}' = \mathbbm{Z}_{< K} \times (\mathbbm{Z}_{\geq 0})^{d-1}$, necessitating an alternative approach. This geometry also means that we are unable to apply classical Foster-Lyapunov type results to establish the existence of a quasi-stationary distribution \parencite{meyn1993stability, champagnat2023general}, that would otherwise apply if the system was supercritical on a bounded domain only.

\subsection{Preliminaries}

In the ensuing analysis, we are going to need the following constants:
\begin{align}
    \kappa & := \sup_{x_1 \in \mathbbm{Z}_{\geq 0}, i,j \in \{1, \dots, d \}} \gamma_i(x_1) \mu_{i,j}(x_1) \label{constant::xi} \\
    \rho & := \sup_{x_1 \in \mathbbm{Z}_{\geq 0}} \gamma_1(x_1) \label{constant::rho} \\
    \gamma_\text{min} & := \inf_{x_1 \in \mathbbm{Z}_{\geq 0}, i \in \{1, \dots, d\}} \gamma_i(x_1) \label{constant::gamma_min} \\
    \gamma_\text{max} & := \sup_{x_1 \in \mathbbm{Z}_{\geq 0}, i \in \{1, \dots, d\}} \gamma_i(x_1) \label{constant::gamma_max}.
\end{align}

Under Assumption \ref{assumption::type_1_carrying_capacity}, it is necessarily the case that $\gamma_\text{min}>0$.

We pay particular attention to the limiting matrix
\begin{align}
    H := \lim_{x_1 \to \infty} \Gamma(x_1) ( M(x_1) - I_d), \label{eq::limiting_H}
\end{align}
which is well-defined under Assumption \ref{assumption::type_1_carrying_capacity}. By construction, $-H$ is a non-singular $M$-matrix (with non-positive off-diagonal components and eigenvalues with strictly positive real parts). Then by \textcite{plemmons1976survey}, $-H$ is semipositive, and so there exists a strictly positive vector $\boldsymbol{\xi} > 0$ such that $H \boldsymbol{\xi} < 0$. Choose such a vector $\boldsymbol{\xi} > 1$ and then define:
\begin{align}
    h_\text{min} & := \min_{i \in \{1 , \dots, d\}} |H \boldsymbol{\xi}|_{i} \label{constant::g_min} \\
    \xi_\text{min} & := \min_{i \in \{1 , \dots, d\}} \xi_i \label{constant::mu_min} \\
    \xi_\text{max} & := \max_{i \in \{1 , \dots, d\}} \xi_i \label{constant::mu_max}.
\end{align}

Given element-wise convergence $\Gamma(x_1) (M(x_1) - I_d) \to H$ in the limit $x_1 \to \infty$, there exists a threshold $y$ such that, for all $ i,j \in \{1, \dots, d\}$ and $x_1 > y$, $|\gamma_{i}(x_1) (\mu_{i,j}(x_1) - \delta_{i,j}) - H_{i,j}| < \frac{h_\text{min}}{2 \xi_\text{max} d} $. This yields the element-wise inequality
\begin{align}
    \sup_{x_1 > y} \Gamma(x_1) (M(x_1) - I_d) \boldsymbol{\xi} < -\frac{h_\text{min}}{2} \mathbf{1} \label{eq::threshold_y}
\end{align}
which we will use on several occasions. 

\section{Extinction almost surely} \label{sec::extinction_almost_surely}

Given the state $\mathbf{0}$ is necessarily absorbing, while any non-zero state with finitely-many individuals of types $1, \dots, d$ is necessarily transient, it is straightforward to deduce that
\begin{align}
    \mathbbm{P} \Big( \lim_{t \to \infty} \lVert \mathbf{x}(t) \rVert = 0 \Big) +  \mathbbm{P} \Big( \lim_{t \to \infty} \lVert \mathbf{x}(t) \rVert = \infty \Big) = 1. \label{eq::extinction_explosion_dichotomy}
\end{align}

In Theorem \ref{theorem::extinction_almost_surely} below, we prove the stronger result that extinction occurs almost surely under Assumptions \ref{assumption::type_1_carrying_capacity} and \ref{assumption::type_1_offspring}, with the time to extinction possessing finite moments of all orders. Here, we draw on ideas from the proofs of Lemma 4.3 of \textcite{ferrari1995existence}, Theorem 3 of \textcite{gonzalez2005unlimited} and Lemma 6.1.4 of \textcite{anderson2012continuous}.

\begin{theorem} \label{theorem::extinction_almost_surely}
Under Assumptions \ref{assumption::type_1_carrying_capacity} and \ref{assumption::type_1_offspring}, $\lVert \mathbf{x}(t) \rVert \to 0$ almost surely as $t \to \infty$. Further, the time to extinction
\begin{align*}
    E_{\mathbf{s}} := \inf\{ t: \mathbf{x}(t) = \mathbf{0} \, | \, \mathbf{x}(0) = \mathbf{s} \}
\end{align*}
possesses finite moments of all orders, that is, 
\begin{align*}
    \EX[(E_\mathbf{s})^r ] < \infty \text{ for all } \mathbf{s} \in (\mathbbm{Z}_{\geq 0})^d \text{ and } r \in \mathbbm{N}.
\end{align*}
\end{theorem}

\begin{proof}

To prove that extinction occurs almost surely given (\ref{eq::extinction_explosion_dichotomy}), it suffices to show that 
\begin{align*}
    \mathbbm{P} \Big( \liminf_{t \to \infty} \lVert \mathbf{x}(t) \rVert = \infty \Big) = 0.
\end{align*}

To this end, we show that the process $\mathbf{x}$ initialised in any state almost surely hits a suitably-constructed finite domain $\mathcal{C} \supset \{ \mathbf{0} \}$. We construct this domain $\mathcal{C}$ as follows.

Applying the inequality (\ref{eq::threshold_y}), we can establish the existence of a threshold $y$ and an element-wise positive vector $\boldsymbol{\xi} >0$ such that, for all $x_1 > y$,
\begin{align}
    \sum^d_{i=1} \Gamma_i(x_1) x_i \sum^d_{j=1} \big[ \mu_{i,j}(x_1) - \delta_{i,j }] \xi_j < -\frac{h_\text{min}}{2 \xi_\text{max}}  \mathbf{x} \boldsymbol{\xi} \label{eq::threshold_y_inequality}
\end{align}
where the constant $h_\text{min} > 0$ is given by Equation (\ref{constant::g_min}) and $\xi_\text{max} > 0$ is the maximal component of $\boldsymbol{\xi}$ (Equation (\ref{constant::mu_max})). For notational shorthand, we set
\begin{align}
    q := \frac{h_\text{min}}{2 \xi_\text{max}}. \label{const::q}
\end{align}

Using Theorem 8.2 of \textcite{jagers2016size}, which is essentially an application of Dynkin's formula,
\begin{align}
    \mathbf{x}(t) \boldsymbol{\xi} =  \mathbf{x}(0) \boldsymbol{\xi} + \int^t_0  \sum^d_{i=1} \Gamma_i \big( x_1(s) \big) x_i(s) 
 \sum^d_{j=1}  \big[ \mu_{i,j} \big( x_1(s) \big) - \delta_{i,j }] \xi_j ds + R_t \label{eq::dynkin_formula}
\end{align}
where $R_t$ is a local martingale with mean zero and predictable quadratic variation. Observing that uniformly bounded local martingales are martingales, following \parencite{jagers2011population}, we take the localising sequence $L_k = \inf\{ \tau: \mathbf{x}(\tau) \boldsymbol{\xi} > k \}$ so $\EX[ R_{t \wedge L_k}] = 0$ for all $t \geq 0$ and $k \geq 0$. Since the process $\mathbf{x}$ is regular and non-explosive, $L_k \to \infty$ almost surely in the limit $k \to \infty$.

Equations (\ref{eq::threshold_y_inequality}) and (\ref{eq::dynkin_formula}) collectively imply that the expectation $\EX[\mathbf{x}(t) \boldsymbol{\xi}]$ decays in the subcritical regimes attained when the type 1 subpopulation $x_1(t) \ge y$; but may grow in the possibly supercritical regimes attained when $x_1(t) < y$. However, when the types $2, \dots, d$ subpopulations are large, Assumption \ref{assumption::type_1_offspring} yields a meaningful lower bound $\varepsilon \gamma_\text{min} \sum^d_{i=2} x_i$ (where the constant $\gamma_\text{min}$ is given by Equation (\ref{constant::gamma_min})) on the rate at which type 1 offspring are generated, and thus the duration of time that the process can spend in the possibly supercritical regimes with $x_1(t) < y$. 

We therefore construct the finite domain $\mathcal{C}$ to encompass states with $x_1(t) < y$ and $\sum^d_{i=2} x_i(t) < z$ for some sufficiently large threshold $z$. To prove that the process hits the domain $\mathcal{C}$ almost surely, we analyse the upcrossings and downcrossings of the process over the type 1 subpopulation threshold $x_1 = y$ prior to hitting the domain $\mathcal{C}$. Provided $z$ is sufficiently large, we show that the expected growth of $\mathbf{x}(t) \boldsymbol{\xi}$ in the possibly supercritical regimes with $x_1 < y$ and $\sum^d_{i=2} x_i > z$ is dominated by the expected decay in the subcritical regimes attained when $x_1 \geq y$. We also generate bounds on moments of all orders for the upcrossing and downcrossing times, which we then use to bound first passage times to the domain $\mathcal{C}$ and finally the time to extinction.

We structure the proof into a series of lemmas:
\begin{itemize}
    \item Lemma \ref{lemma::downcrossing} bounds the moment generating function (MGF) of the downcrossing time $T_y$ over the type 1 threshold $x_1 = y$, in addition to the expectation $\EX[\mathbf{x}(T_y) \boldsymbol{\xi}]$. 
    \item Lemma \ref{lemma::upcrossing} bounds the MGF for the time $D_y^z$ to exit the domain $\mathcal{D}^z_y = \{ \mathbf{x} \in (\mathbbm{Z}_{\geq 0})^d: \sum^d_{i=2} x_i>z, x_1(0) < y \}$, in addition to the expectation $\EX[\mathbf{x}(D^z_y) \boldsymbol{\xi}]$.
    \item Lemma \ref{lemma::first_passage_time} shows that the first passage time to a finite domain of the form $\mathcal{C} = \{ \mathbf{x} \in (\mathbbm{Z}_{\geq 0})^d: \sum^d_{i=2} x_i \leq z', x_1(0) < y \}$ possesses finite moments of all orders $r \in \mathbbm{N}$, whereby $\mathbbm{P}(\liminf_{t \to \infty} \mathbf{x}(t) \boldsymbol{\xi} = \infty) = 0$; coupled with the extinction-explosion dichotomy (\ref{eq::extinction_explosion_dichotomy}), this allows us to establish that extinction occurs almost surely.
    \item Lemma \ref{lemma::extinction_moments} bounds moments of all orders $r \in \mathbbm{N}$ for the time to extinction.
\end{itemize}


\begin{lemma}[Downcrossings over the type 1 subpopulation threshold $x_1=y$] \label{lemma::downcrossing}
Suppose we initialise $\mathbf{x}(0)$ such that $x_1(0) \geq y$. Define the downcrossing time
\begin{align*}
    T_y := \inf\{ t: x_1(t) < y\}.
\end{align*}
There exists a constant $\delta^* > 0$ such that
\begin{align*}
     \EX \big[ e^{\delta T_y} \big] \leq \frac{\mathbbm{E}[ \mathbf{x}(0) \boldsymbol{\xi}]}{\xi_\text{min} (y-1)} \text{ for all } \delta < \delta^*.
\end{align*}
where the constant $\xi_\text{min} >0$ is given by Equation (\ref{constant::mu_min}). Further,
\begin{align*}
    \EX \big[ \mathbf{x}(T_y) \boldsymbol{\xi} \big] \leq \frac{\rho y}{\rho y + q} \EX \big[ \mathbf{x}(0) \boldsymbol{\xi} \big]
\end{align*}
where the constants $\rho > 0$ and $q > 0$ are given by Equations (\ref{constant::rho}) and (\ref{const::q}) respectively. 
\end{lemma}

\begin{subproof}
 
To bound the MGF of $T_y$, we draw on ideas from the proof of Lemma 4.3 of \textcite{ferrari1995existence}. Let $\sigma_k$ denote the time of the $k^\text{th}$ jump, with $\sigma_0 = 0$ and 
\begin{align*}
    \nu := \inf\{k: \sigma_k \geq T_y \}.
\end{align*}
For $1 \leq k \leq \nu$, the inequality (\ref{eq::threshold_y}) implies that
\begin{align}
    \EX \Bigg[ \frac{\mathbf{x} \big( \sigma_{k} \big) \boldsymbol{\xi}}{\mathbf{x} \big( \sigma_{k-1} \big) \boldsymbol{\xi}} \, \Bigg| \, \mathbf{x} \big( \sigma_{k-1} \big) \boldsymbol{\xi} \Bigg] \leq 1 - \frac{h_\text{min}}{2 \gamma_\text{max}} \frac{1}{\mathbf{x} \big( \sigma_{k-1} \big) \boldsymbol{\xi}} \label{eq::eigenvalue_bound}
\end{align}
where the constant $\gamma_\text{max}$ is defined in Equation (\ref{constant::gamma_max}). 

For $1 \leq k \leq \nu$, we additionally observe that the random variable
\begin{align*}
    W_{k-1} \sim \text{Exponential}( \mathbf{x}( \sigma_{k-1}) \boldsymbol{\xi} \cdot \gamma_\text{min}/\xi_\text{max})
\end{align*}
exhibits first order stochastic dominance over the inter-jump intervals $(\sigma_{k} - \sigma_{k-1})$, because the instantaneous death rates for each individual are uniformly below by $\gamma_\text{min}$. Set
\begin{align*}
    \delta^* = \frac{h_\text{min} \gamma_\text{min}}{2 \gamma_\text{max} \xi_\text{max}} \wedge \frac{\gamma_\text{min} \xi_\text{min} y}{\xi_\text{max}}.
\end{align*}
Then for any $\delta < \delta^*$ and $1 \leq k \leq \nu$, Equation (\ref{eq::eigenvalue_bound}) implies that
\begin{align}
    \EX \Bigg[ e^{\delta (\sigma_{k} - \sigma_{k-1})}  \frac{\mathbf{x} \big( \sigma_{k} \big) \boldsymbol{\xi}}{\mathbf{x} \big( \sigma_{k-1} \big) \boldsymbol{\xi}} \, \Bigg| \, \mathbf{x}(\sigma_{k-1}) \Bigg] & \leq  \EX \big[ e^{\delta W_{k-1}} \big] \EX \Bigg[ \frac{\mathbf{x} \big( \sigma_{k} \big) \boldsymbol{\xi}}{\mathbf{x} \big( \sigma_{k-1} \big) \boldsymbol{\xi}} \, \Bigg| \, \mathbf{x}(\sigma_{k-1}) \Bigg] < 1. \label{eq::intermediate_bound_downcrossing_time}
\end{align}

Therefore, for any $\delta < \delta^*$ the process
\begin{align*}
    z(n) := e^{\delta \sigma_{n \wedge \nu} } \frac{\mathbf{x} \big( \sigma_{n \wedge \nu} \big) \boldsymbol{\xi}}{\mathbf{x} (0) \boldsymbol{\xi}},  n \geq 0
\end{align*}
constitutes a non-negative super-martingale with respect to the filtration $\mathcal{F}_n = \sigma( \mathbf{x} (0), \dots, \mathbf{x} (\sigma_{n \wedge \nu}))$. By the martingale convergence theorem (Corollary 11.5 of \parencite{klenke2008probability}), it then follows that
\begin{align*}
    \EX \big[ e^{\delta T_y} \mathbf{x}(T_y) \boldsymbol{\xi} \big] \leq \mathbbm{E}[ \mathbf{x}(0) \boldsymbol{\xi}] \text{ for all } \delta < \delta^*.
\end{align*}

Since each transition involves the death of at most one individual, $\mathbf{x}(T_y) \boldsymbol{\xi} \geq \xi_\text{min}(y-1)$ almost surely, yielding the bound
\begin{align*}
     \EX \big[ e^{\delta T_y} \big] \leq \frac{\mathbbm{E}[ \mathbf{x}(0) \boldsymbol{\xi}]}{\xi_\text{min} (y-1)} \text{ for all } \delta < \delta^*.
\end{align*}

Next, we generate a tighter upper bound on $\EX [ \mathbf{x}(T_y) \boldsymbol{\xi} ]$. Using Equations (\ref{eq::threshold_y_inequality}) and (\ref{eq::dynkin_formula}), and the optional stopping theorem (Theorem 3.2 of \parencite{revuz2013continuous}), we obtain
\begin{align*}
    \EX \big[ \mathbf{x}(t \wedge T_y \wedge L_k ) \boldsymbol{\xi} \big] &= \EX \big[ \mathbf{x}(0) \boldsymbol{\xi} \big] + \EX \bigg[ \int^{t \wedge T_y \wedge L_k}_0   \sum^d_{i=1} \Gamma_i \big( x_1(s) \big) x_i(s) \sum^d_{j=1}  \big[ \mu_{i,j} \big( x_1(s) \big) - \delta_{i,j }] \xi_j ds \bigg] \\
    & \leq \EX \big[ \mathbf{x}(0) \boldsymbol{\xi} \big] - q \EX \Big[ \int^{t \wedge T_y \wedge L_k}_0  \mathbf{x}(s \wedge T_y \wedge L_k) \boldsymbol{\xi} ds \Big].
\end{align*}

We then apply Fubini's theorem and a variation of Grownwall's inequality \parencite{pham2007variation} to conclude
\begin{align}
    \EX \big[ \mathbf{x}(t \wedge T_y \wedge L_k ) \boldsymbol{\xi} \big] \leq \EX \big[ \mathbf{x}(0) \boldsymbol{\xi} e^{-q (t \wedge T_y \wedge L_k)} \big]. \label{eq::gronwall_bound_1}
\end{align}

An application of Ville's maximal inequality (Exercise 1.15 of \parencite{revuz2013continuous}) to the supermartingale $\{ \mathbf{x}(\sigma_{n \wedge \nu}): n \in \mathbbm{Z}_{\geq 0} \}$ yields
\begin{align}
    \mathbbm{P} \Big( \sup_{0 \leq t \leq T_y} \mathbf{x}(t) \boldsymbol{\xi} \geq k \Big) \leq \frac{\EX[\mathbf{x}(0) \boldsymbol{\xi}]}{k} \implies \mathbbm{P}( T_y \geq L_k) \leq \frac{\EX[\mathbf{x}(0) \boldsymbol{\xi}]}{k}. \label{eq::ville_inequality_bound}
\end{align}

Using Fatou's lemma and the bounds (\ref{eq::gronwall_bound_1}) and (\ref{eq::ville_inequality_bound}), it then follows that
\begin{align}
    \EX \big[  \mathbf{x}(T_y \wedge t) \boldsymbol{\xi} \big]  &= \EX \big[ \liminf_{k \to \infty} \mathbf{x}(t \wedge T_y \wedge L_k ) \boldsymbol{\xi} \big] \notag \\
    & \leq \liminf_{k \to \infty} \EX \big[ \mathbf{x}(t \wedge T_y \wedge L_k ) \boldsymbol{\xi} \big] \notag \\
    & \leq \liminf_{k \to \infty} \Big\{ \EX \big[ \mathbf{x}(0) \boldsymbol{\xi} e^{-q (t \wedge T_y)} \, \big| \, T_y < L_k \big] \mathbbm{P}(T_y < L_k) \notag \\
    & \qquad \qquad \qquad + \EX \big[ \mathbf{x}(0) \boldsymbol{\xi} e^{-q (t \wedge L_k)} \, \big| \, T_y \geq L_k \big] \mathbbm{P}(T_y \geq L_k) \Big\} \notag \\
    & \leq  \EX \big[ \mathbf{x}(0) \boldsymbol{\xi} e^{-q (t \wedge T_y)} \big] + \EX \big[ \mathbf{x}(0) \boldsymbol{\xi} \big]  \lim_{k \to \infty} \mathbbm{P}(T_y \geq L_k) \notag \\
    & = \EX \big[ \mathbf{x}(0) \boldsymbol{\xi} e^{-q (t \wedge T_y)} \big] \label{eq::downcrossing_intermediate_bound}.
\end{align}

Now, suppose we `mark' $y$ of the initial $x_1(0)$ type 1 individuals in the process. In the interval $[0, T_y]$, at least one marked type 1 individual dies almost surely. Therefore, $T_y$ exhibits first order stochastic dominance over the time until one marked type 1 individual dies, which in turn exhibits first order stochastic dominance over the random variable $V_y \sim \text{Exponential}(\rho y)$, where the constant $\rho$ (Equation (\ref{constant::rho})) is an upper bound for the death rate of type 1 individuals. Then using Equation (\ref{eq::downcrossing_intermediate_bound}), Fatou's lemma and the dominated convergence theorem, it thus follows that
\begin{align*}
     \EX \big[ \mathbf{x}(T_y) \boldsymbol{\xi} \, \big| \, x_1(0) \geq y \big] & = \EX \Big[ \liminf_{t \to \infty} \mathbf{x}(T_y \wedge t) \boldsymbol{\xi} \, \big| \, x_1(0) \geq y \Big] \notag \\
     & \leq \liminf_{t \to \infty} \EX \big[  \mathbf{x}(T_y \wedge t) \boldsymbol{\xi} \, \big| \, x_1(0) \geq y \big] \notag \\
     & \leq \liminf_{t \to \infty} \EX \big[  \mathbf{x}(0) \boldsymbol{\xi} e^{-q (T_y \wedge t)} \, \big| \, x_1(0) \geq y \big] \notag \\
     & \leq \liminf_{t \to \infty} \EX \big[  \mathbf{x}(0) \boldsymbol{\xi} e^{-q (V_y \wedge t)} \big] \notag \\
    & = \frac{\rho y}{\rho y + q} \EX \big[ \mathbf{x}(0) \boldsymbol{\xi} \big],
\end{align*}
as claimed.
\end{subproof}


\begin{lemma}[Upcrossings over the type 1 subpopulation threshold $x_1 = y$ while $\sum^d_{i=2} x_i > z$] \label{lemma::upcrossing}
Consider the domain
\begin{align*}
    \mathcal{D}_y^z := \Big\{ \mathbf{x} \in (\mathbbm{Z}_{\geq 0})^d: \sum^d_{i=2} x_i > z \text{ and } x_1(0) < y \Big\}.    
\end{align*}
Suppose we initialise $\mathbf{x}(0) \in \mathcal{D}_y^z$. Define the stopping time 
\begin{align*}
    D_y^z := \inf\{ t:  \mathbf{x}(t) \notin \mathcal{D}_y^z \}
\end{align*}
and the function
\begin{align*}
    f^z_y(v) = \sum^y_{r=0} {y \choose r} \bigg(-\frac{v}{\rho} \bigg)^{(r)} \bigg(\frac{\rho}{\gamma_\text{min} \varepsilon z}\bigg)^r 
\end{align*}
where $(\cdot)^{(\cdot)}$ denotes the rising factorial and the constants $\rho$ and $\gamma_\text{min}$ are given by Equations (\ref{constant::rho}) and (\ref{constant::gamma_min}) respectively. There exists a threshold $z^* > 0$ such that for all $z \geq z^*$
\begin{align*}
    \EX \big[ e^{\kappa \xi_\text{max} D^{z}_y} \big] &\leq \big( f_y^{z}(\kappa \xi_\text{max}) \big)^{-1}\\
    \EX \big[ \mathbf{x}(D_y^{z}) \boldsymbol{\xi} \big] & \leq \EX \big[ \mathbf{x}(0) \boldsymbol{\xi} \big] \big( f_y^{z}(\kappa \xi_\text{max}) \big)^{-1}
\end{align*}
where the constants $\kappa$ and $\xi_\text{max}$ are given by Equations (\ref{constant::xi}) and (\ref{constant::mu_max}) respectively.
\end{lemma}

\begin{subproof}
    Following \parencite{jagers2011population}, we use Equation (\ref{eq::dynkin_formula}) and the optional stopping theorem (Theorem 3.2 of \parencite{revuz2013continuous}) to obtain the bound
\begin{align*}
    \EX \big[ \mathbf{x}(t \wedge D_y^z \wedge L_k) \boldsymbol{\xi} \big] &= \EX \big[ \mathbf{x}(0) \boldsymbol{\xi} \big] + \EX \bigg[ \int^{t \wedge D_y^z \wedge L_k}_0   \sum^d_{i=1} \Gamma_i \big( x_1(s) \big) x_i(s) \sum^d_{j=1}  \big[ \mu_{i,j} \big( x_1(s) \big) - \delta_{i,j }] \xi_j ds \bigg] \\
    &\leq \EX \big[ \mathbf{x}(0) \boldsymbol{\xi} \big] + \kappa \xi_\text{max} \cdot \mathbbm{E} \Bigg[ \int^{t \wedge D_y^z \wedge L_k}_0 \mathbf{x}(s \wedge D_y^z \wedge L_k) \boldsymbol{\xi} ds \Bigg].
\end{align*}

By Fubini's theorem and Gronwall's inequality, it then follows that
\begin{align*}
    \EX \big[ \mathbf{x}(t \wedge D_y^z \wedge L_k) \boldsymbol{\xi} \big] & \leq \EX \big[ \mathbf{x}(0) \boldsymbol{\xi} e^{\kappa \xi_\text{max} (D_y^z \wedge L_k \wedge t)} \big] \leq \EX \big[ \mathbf{x}(0) \boldsymbol{\xi} e^{\kappa \xi_\text{max} (D_y^z \wedge t)} \big],
\end{align*}
whereby the dominated convergence theorem yields, 
\begin{align}
    \EX \big[ \mathbf{x}(D_y^z \wedge t) \boldsymbol{\xi} \big] = \EX \Big[ \lim_{k \to \infty} \mathbf{x}(t \wedge D_y^z \wedge L_k) \boldsymbol{\xi}  \Big] \leq \EX\big[\mathbf{x}(0) \boldsymbol{\xi} e^{\kappa \xi_\text{max} (D_y^{z} \wedge t)} \big]. \label{eq::upcrossing_time_bound}
\end{align}

To generate an upper bound for $\EX[ \mathbf{x}(D_y^z \wedge t) \boldsymbol{\xi}]$, it then suffices to bound the MGF of the stopping time $D^z_y$.

Under Assumption \ref{assumption::type_1_offspring}, in the interval $[0, D_y^z]$, individuals of types $2, \dots, d$ generate offspring batches including one or more type 1 offspring at rate $\geq \gamma_\text{min} \varepsilon z$, while the lifetime of each type 1 individual exhibits first order stochastic dominance over the random variable $\text{Exponential}(\rho)$. Therefore, we can construct a coupling of the size of the type 1 subpopulation $x_1(t)$, and the number of busy servers $n_z(t)$ in an $M/M/\infty$ queue, with $n_z(0) = 0$, service rate $\rho$ and homogeneous arrival rate $\gamma_\text{min} \varepsilon z$, such that $n_z(t) \leq x_1(t)$ for all $t \in [0, D^z_y]$. Let
\begin{align*}
    Q_y^z = \inf \{ t: n_z(t) = y \}
\end{align*}
be the first passage time from zero to $y$ busy servers in this $M/M/\infty$ queue. Then $Q^z_y$ exhibits first order stochastic dominance over $D_y^z$.

Using Equation (2.12) of \textcite{morrison1987queueing} and Equation (167) of \parencite{morrison1980analysis}, the MGF for $Q^z_y$ takes the form
\begin{align}
    \EX \big[ e^{v Q^z_y} \big] = \big( f^z_y(v) \big)^{-1}. \label{eq::Q_MGF}
\end{align}
We observe that Equation (\ref{eq::Q_MGF}) is well-defined in the domain $v \in (-\infty, \theta^z_y)$, where $\theta^z_y$ is the smallest (necessarily positive) real root of the polynomial $f^z_y(v)$ \parencite{morrison1980analysis, morrison1987queueing}. From the bound 
\begin{align*}
    f^z_y(v) & = 1 - \frac{v}{\gamma_\text{min} \varepsilon z} \sum^{y-1}_{r=0} {y \choose r + 1} \Big(1 -\frac{v}{\rho} \Big)^{(r)} \Big(\frac{\rho}{\gamma_\text{min} \varepsilon z} \Big)^r \\
    & \geq 1 -  \frac{2v}{\gamma_\text{min} \varepsilon z} \sum^{y-1}_{r=0} {y-1 \choose r} \Big( y + \frac{v}{\rho} \Big)^r \Big( \frac{\rho}{\gamma_\text{min} \varepsilon z} \Big)^r \\
    &= 1 - \frac{2v}{\gamma_\text{min} \varepsilon z} \Big( 1 + \frac{\rho}{\gamma_\text{min} \varepsilon z} \Big[ y + \frac{v}{\rho} \Big] \Big)^{y-1},
\end{align*}
for any $L$, there exists a threshold $z^*$ such that $f^{z}_y(v) > 0$ for all $v \in [0, 2 L]$ and $z \geq z^*$. Taking $L=\kappa \xi_\text{max}$, we conclude that $\theta^{z}_y > \kappa \xi_\text{max}$. Recall that $Q^{z}_y$ exhibits first order stochastic dominance over $D^{z}_y$ and is independent of the initial condition $\mathbf{x}(0)$. Then for all $z \geq z^*$, we obtain the well-defined bound
\begin{align}
    \EX \big[ e^{\kappa \xi_\text{max} (D^{z}_y \wedge t)} \big] \leq \EX \big[ e^{\kappa \xi_\text{max} Q^{z}_y} \big] = \big( f_y^{z}(\kappa \xi_\text{max}) \big)^{-1}. \label{eq::upcrossing_bound_intermediate}
\end{align}

Using Equations (\ref{eq::upcrossing_time_bound}) and (\ref{eq::upcrossing_bound_intermediate}), and applying the dominated convergence theorem, it follows that for all $z \geq z^*$,
\begin{align*}
     \EX \big[ \mathbf{x}(D_y^{z}) \boldsymbol{\xi} \big] &=  \EX \big[ \lim_{t \to \infty}  \mathbf{x}(D_y^{z} \wedge t) \boldsymbol{\xi} \big] = \lim_{t \to \infty} \EX \big[  \mathbf{x}(D_y^{z} \wedge t) \boldsymbol{\xi} \big] \leq \EX \big[ \mathbf{x}(0) \boldsymbol{\xi} \big] \big( f_y^{z}(\kappa \xi_\text{max}) \big)^{-1},
\end{align*}
as claimed.
\end{subproof}


\begin{lemma}[First passage time for a finite domain $\mathcal{C} \supset \{ \mathbf{0} \}$] \label{lemma::first_passage_time}
Consider the domain
\begin{align*}
    \mathcal{C}^{z}_y := \Big\{ \mathbf{x} \in (\mathbbm{R}_{\geq 0})^d: x_1 < y \text{ and } \sum^d_{i=2} x_i \leq z \Big\}.
\end{align*}
Suppose we initialise $\mathbf{x}(0) \in (\mathbbm{Z}_{\geq 0})^d \setminus \mathcal{C}^{z}_y$. Define the stopping time
\begin{align*}
    Z^{z}_y := \inf \{t: \mathbf{x}(t) \in \mathcal{C}^{z}_y \}.
\end{align*}
Then for sufficiently large  $z' > 0$, for each $r \in \mathbbm{N}$, there exist constants $b_r^{(1)}, b_r^{(2)} < \infty$ such that
\begin{align*}
    \EX \big[ (Z^{z'}_y)^r  \big] \leq b_r^{(1)} \mathbf{x}(0) \boldsymbol{\xi} + b_r^{(2)} \big( \mathbf{x}(0) \boldsymbol{\xi} \big)^2. 
\end{align*}
\end{lemma}

\begin{subproof}
Choose $z' \geq z^*$ to be sufficiently large to ensure that
\begin{align*}
    \eta: = \big( f_y^{z'}(\kappa \xi_\text{max}) \big)^{-1} \cdot \frac{\rho y}{\rho y + q} < 1.
\end{align*}
(this is possible because $f_y^{z}(\kappa \xi_\text{max}) \to 1$ in the limit $z \to \infty$).

Drawing inspiration from the proof of Theorem 3 of \textcite{gonzalez2005unlimited}, we construct a modified process $\mathbf{x}^*(t)$ that is immediately absorbed when the branching process $\mathbf{x}(t)$, initialised with $x_1(0) \geq y$, enters the finite domain $\mathcal{C}^{z'}_y$, that is,
\begin{align*}
    \mathbf{x}^*(t) = \begin{cases}
    \mathbf{x}(t) & \text{ if } t < Z^{z'}_y \\
    0 & \text{ if } t \geq Z^{z'}_y. 
    \end{cases}
\end{align*}

We consider a sequence of stopping times on the interval $[0, Z^{z'}_y)$ based on the upcrossings and downcrossings of the type 1 subpopulation in the modified process $\mathbf{x}^*(t)$ over the threshold $y$. Setting $D_y^{(0)} = 0$, on the interval $[0, Z^{z'}_y)$, we recursively define
\begin{align*}
    T_y^{(\ell)} & := \inf \Big\{ t > D_y^{(\ell-1)}: x^*_1(t) < y \Big\} \\
    D_y^{(\ell)} & := \inf \Big\{ t > T_y^{(\ell)}: x^*_1(t) \geq y \Big\}.
\end{align*}

We also define the respective number of downcrossings and upcrossings of the type 1 subpopulation threshold $x_1=y$
\begin{align*}
    N_\text{down} & := \sup \{ k: T_y^{(k)} < Z^{z'}_y\} \\
    N_\text{up} & := \sup \{ k: D_y^{(k)} < Z^{z'}_y\}.
\end{align*}

For all $\ell > N_\text{down}$, we set $T_y^{(\ell)} = Z^{z'}_y$; likewise, for all $\ell > N_\text{up}$, we set $D_y^{(\ell)} = Z^{z'}_y$. 

Using Lemmas \ref{lemma::downcrossing} and \ref{lemma::upcrossing},
\begin{align*}
    \EX \big[ \mathbf{x}^* \big( D_y^{(n)} \big) \boldsymbol{\xi}\big] & \leq \big( f_y^{z'}(\kappa \xi_\text{max}) \big)^{-1} \EX \big[ \mathbf{x}^* \big( T_y^{(n)} \big) \boldsymbol{\xi} \big] \\
    & \leq \big( f_y^{z'}(\kappa \xi_\text{max}) \big)^{-1} \cdot \frac{\rho y}{\rho y + q} \EX \big[ \mathbf{x}^* \big( D_y^{(n-1)} \big) \boldsymbol{\xi} \big] = \eta \EX \big[ \mathbf{x}^* \big( D_y^{(n-1)} \big) \boldsymbol{\xi} \big].
\end{align*}

By induction, for all $n \geq 0$, it follows that
\begin{align}
    \EX \big[ \mathbf{x}^* \big( D_y^{(n)} \big) \boldsymbol{\xi}\big] \leq \eta^n  \mathbf{x} (0) \boldsymbol{\xi}. \label{eq::geometric_bound_hitting}
\end{align}

We now bound the moments of the stopping time $Z^{z'}_y$. We begin by observing that Equation (\ref{eq::geometric_bound_hitting}) yields a geometric bound for the number of upcrossings of the type 1 subpopulation threshold $x_1 = y$, until the process hits the finite domain $\mathcal{C}^{z'}_y \supset \{ 0 \}$. This is because
\begin{align*}
    \mathbf{x}^* \big( D_y^{(n)} \big) \boldsymbol{\xi} \leq z' \xi_\text{min} \implies \mathbf{x}^*\big( D_y^{(n)} \big) \in \mathcal{C}^{z'}_y \implies D_y^{(n)} = Z^{z'}_y,
\end{align*}
whereby Markov's inequality yields
\begin{align}
    \mathbbm{P}(N_\text{up} \geq n) = \mathbbm{P} \big( D_y^{(n)} < Z^{z'}_y \big) \leq \mathbbm{P} \Big( \mathbf{x}^* \big( D_y^{(n)} \big) \boldsymbol{\xi} > z' \xi_\text{min}  \Big) \leq \frac{\EX[\mathbf{x}^* \big( D_y^{(n)} \big) \boldsymbol{\xi}]}{z' \xi_\text{min} } \leq \eta^n \frac{ \mathbf{x} (0) \boldsymbol{\xi} }{z' \xi_\text{min} }. \label{eq::geometric_upcrossings}
\end{align}

We then bound each upcrossing and downcrossing interval. Set $\alpha^* = \kappa \xi_\text{max} \wedge \frac{\delta^*}{2}$. For all $\ell > 1$, the random variable $Q^{z'}_y$, with MGF (\ref{eq::Q_MGF}) well-defined for parameters $\alpha \leq \kappa \xi_\text{max}$, exhibits first order stochastic dominance over the upcrossing intervals $(D_y^{(\ell)} - T_y^{(\ell)})$. We thus conclude that
\begin{align*}
    \EX \big[ e^{\alpha (D_y^{(\ell)} - T_y^{(\ell)})} \big] \leq (f^{z'}_y( \alpha))^{-1} < \infty \text{ for all } \alpha \leq \kappa \xi_\text{max},
\end{align*}
yielding the moment bound
\begin{align}
    \EX \big[ (D_y^{(\ell)} - T_y^{(\ell)})^r \big] \leq \frac{r!}{ (\alpha^*)^r f^{z'}_y( \alpha^*)} \text{ for all } r \in \mathbbm{N}. \label{eq::upcrossing_interval_MGF_final}
\end{align}

Additionally, for $\ell \leq N_\text{up}$ (that is, prior to hitting the finite domain $\mathcal{C}^{z'}_y$), we can bound MGFs for the downcrossing intervals $(T_y^{(\ell+1)} - D_y^{(\ell)})$ using Lemma \ref{lemma::downcrossing}:
\begin{align*}
     \EX \big[ e^{\alpha (T_y^{(\ell+1)} - D_y^{(\ell)})} \, \big| \, \ell \leq N_\text{up} \big] \leq \frac{\mathbbm{E}[ \mathbf{x}(D_y^{(\ell)}) \boldsymbol{\xi} \, | \, \ell \leq \ell_2]}{\xi_\text{min} (y-1)} \leq \frac{\eta^\ell  \mathbf{x}(0) \boldsymbol{\xi}}{\xi_\text{min} (y-1)} \text{ for all } \alpha < \delta^*,
\end{align*}
similarly yielding the moment bound
\begin{align}
     \EX \big[  (T_y^{(\ell+1)} - D_y^{(\ell)})^r \, \big| \, \ell \leq N_\text{up} \big] \leq \frac{r!}{(\alpha^*)^r} \frac{\eta^\ell \mathbf{x}(0) \boldsymbol{\xi}}{\xi_\text{min} (y-1)} \label{eq::downcrossing_interval_MGF_final} \text{ for all } r \in \mathbbm{N}.
\end{align}

Using Equations (\ref{eq::upcrossing_interval_MGF_final}) and (\ref{eq::downcrossing_interval_MGF_final}), and Jensen's inequality, we can then bound the $r^\text{th}$ moment of the hitting time $Z^{z'}_y$ conditioned on the event $\{ N_\text{up} = n \}$ (that is, $n$ upcrossings of the type 1 subpopulation threshold $x_1=y$ prior to hitting the finite domain $\mathcal{C}^{z'}_y$):
\begin{align}
    \EX \big[ (Z^{z'}_y)^r \, | \, N_\text{up} = n \big] = & \EX \bigg[ \bigg( \sum^n_{i=1} \big[ D_y^{(i)} - T_y^{(i)} \big] + \sum^n_{i=1} \big[ T_y^{(i)} - D_y^{(i-1)} \big] \bigg)^r \, \bigg| \, N_\text{up} = n  \bigg] \notag \\
    \leq & (2n)^{r-1} \sum^n_{i=1} \Big\{  \EX \big[ ( D_y^{(i)} - T_y^{(i)} )^r \big] + \EX \big[ ( T_y^{(i)} - D_y^{(i-1)} )^r \, \big| \, i \leq N_\text{up} \big] \Big\} \notag \\
    \leq &  n^r \cdot \frac{2^{r-1} r!}{(\alpha^*)^r} \bigg\{ \frac{1}{f^{z'}_y( \alpha^*)} + \frac{\mathbf{x}(0) \boldsymbol{\xi}}{\xi_\text{min} (y-1)} \frac{1}{1-\eta} \bigg\}. \label{eq::intermediate_moment_bound}
\end{align}

Applying the law of total expectation, and using both Equation (\ref{eq::intermediate_moment_bound}) and the geometric bound (\ref{eq::geometric_upcrossings}) on the number of upcrossings $N_\text{up}$, it then follows that
\begin{align}
    \EX \big[ (Z^{z'}_y)^r  \big] = & \sum^\infty_{n=1} \EX \big[ (Z^{z'}_y)^r \, | \, N_\text{up} = n \big] \mathbbm{P} (N_\text{up} = n) \notag \\
    \leq & \frac{2^{r-1} r!}{(\alpha^*)^r} \bigg\{ \frac{1}{f^{z'}_y( \alpha^*)} + \frac{\mathbf{x}(0) \boldsymbol{\xi}}{\xi_\text{min} (y-1)} \frac{1}{1-\eta} \bigg\} \EX \big[ (N_\text{up})^r \big] \notag \\
    \leq & \frac{2^{r-1} r!}{(\alpha^*)^r} \bigg\{ \frac{1}{f^{z'}_y( \alpha^*)} + \frac{\mathbf{x}(0) \boldsymbol{\xi}}{\xi_\text{min} (y-1)} \frac{1}{1-\eta} \bigg\} \sum^\infty_{n=0} r n^{r-1} \mathbbm{P}(N_\text{up} \geq n) \notag \\
    \leq & \frac{r 2^{r-1} r!}{(\alpha^*)^r} \frac{\mathbf{x}(0) \boldsymbol{\xi}}{z' \xi_\text{min}} \bigg\{ \frac{1}{f^{z'}_y( \alpha^*)} + \frac{\mathbf{x}(0) \boldsymbol{\xi}}{\xi_\text{min} (y-1)} \frac{1}{1-\eta} \bigg\} \sum^\infty_{n=0} n^{r-1} \eta^n < \infty, \label{eq::zy_r_moment_bound}
\end{align}
where we recall that $\eta < 1$. By Remark 2.6 of \textcite{rebenich2018analog}, the upper bound in Equation (\ref{eq::zy_r_moment_bound}) is of order $O((r!)^2 (-2 \log(\eta)/\alpha^*)^r)$ as a function of $r$ in the limit $r \to \infty$. 

We can replicate this argument to establish a similar bound for $\EX[(Z^{z'}_y)^r]$, given an initial condition $\mathbf{x}(0)$ such that $x_1(0) < y$ and $\sum^d_{i=2} x_i > z'$. 

\end{subproof}


\begin{lemma}[Time to extinction] \label{lemma::extinction_moments}
For all $r \in \mathbbm{N}$ and $\mathbf{s} \in (\mathbbm{Z}_{\geq 0})^d$, $\EX[ (E_\mathbf{s})^r ] < \infty$.
\end{lemma}

\begin{subproof}
By Lemma \ref{lemma::first_passage_time}, $\mathbbm{P}(\liminf_{t \to \infty} \lVert \mathbf{x}(t) \rVert = \infty) = 0$. Applying the extinction-explosion dichotomy (\ref{eq::extinction_explosion_dichotomy}), it follows that $\lVert \mathbf{x}(t) \rVert \to 0$ almost surely as $t \to \infty$. To bound moments for the time to extinction, we draw on ideas from the proof of Lemma 6.1.4 of \textcite{anderson2012continuous}.

Since extinction occurs almost surely, for each $\mathbf{y} \in (\mathbbm{Z}_{\geq 0})^d$, there exists a finite sequence of jumps of length $S(\mathbf{y})$ such that $\upsilon(\mathbf{y}) := q_{\mathbf{y}, \mathbf{y}_1} q_{\mathbf{y}_1, \mathbf{y}_2} \dots q_{\mathbf{y}_{S(\mathbf{y})-1}, \mathbf{0}} > 0$ where $q_{\mathbf{y}, \mathbf{y}'}$ denotes the transition rate from state $\mathbf{y}$ to $\mathbf{y}'$. Set $S = \max_{\mathbf{y} \in \mathcal{C}^{z'}_y} S(\mathbf{y})$ and $\upsilon = \min_{\mathbf{y} \in \mathcal{C}^{z'}_y} \upsilon(\mathbf{y}) > 0$ where the domain $\mathcal{C}^{z'}_y$ is constructed as in Lemma \ref{lemma::first_passage_time}.

Suppose we initialise $\mathbf{x}(0) = \mathbf{y} \in \mathcal{C}^{z'}_y$. Denote by $\beta_{k, \mathbf{y}}$ the time of the $k^\text{th}$ jump with $\beta_{0, \mathbf{y}} = 0$ provided $\mathbf{x}(0) = \mathbf{y}$ where, to avoid trivial exceptions, we permit `jumps' within the absorbing state $\mathbf{0}$ at rate $\gamma_\text{min}$. By construction,
\begin{align*}
    \inf_{\mathbf{y} \in \mathcal{C}^{z'}_y} \mathbbm{P} \big( \mathbf{x}(\beta_{S, \mathbf{y}}) = \mathbf{0} \big) \geq \upsilon > 0.
\end{align*}
We additionally observe that the Gamma-distributed random variable with rate parameter $\gamma_\text{min}$ and shape parameter $S$ exhibits first order stochastic dominance over $\beta_{S, \mathbf{y}}$, whereby
\begin{align}
    \EX \big[ (\beta_{S, \mathbf{y}})^r \big] \leq \frac{\Gamma(S + r)}{\Gamma(S) (\gamma_\text{min})^r} \text{ for all } \mathbf{y} \in (\mathbbm{Z}_{\geq 0})^d. \label{eq::bs_r_moment_bound}
\end{align}

Let $E_\mathbf{y}$ denote the first hitting time for the state $\mathbf{0}$ under the condition $\mathbf{x}(0) = \mathbf{y} \in \mathcal{C}^{z'}_y$. On the interval $[0, E_\mathbf{y})$, we define a recursive sequence of stopping times with $B_{0, \mathbf{y}} = 0$ and
\begin{align*}
    B_{\ell, \mathbf{y}} & = \inf \big\{ t \geq B_{\ell-1, \mathbf{y}} + \beta_{S, \mathbf{x}(B_{\ell-1, \mathbf{y}})}: \mathbf{x}(t) \in \mathcal{C}^{z'}_y \big\}.
\end{align*}

Set $U_\mathbf{y} := \sup \{ \ell: B_{\ell, \mathbf{y}} < E_\mathbf{y} \}$. For all $\ell > U_\mathbf{y}$, we set $B_{\ell, \mathbf{y}} = E_\mathbf{y}$. By construction,
\begin{align}
    \mathbbm{P} \Big( \mathbf{x} \big( B_{\ell-1, \mathbf{y}} + \beta_{S, \mathbf{x}(B_{\ell-1}, \mathbf{y})} \big) = \mathbf{0}  \Big) \geq \upsilon \implies \mathbbm{P}( U_\mathbf{y} \geq \ell ) \leq (1 - \upsilon)^\ell \text{ for all } \mathbf{y} \in \mathcal{C}^{z'}_y. \label{eq::uy_geometric_bound}
\end{align}

By Lemma \ref{lemma::first_passage_time}, there exist constants $b^{(1)}_r, b^{(2)}_r < \infty$ such that
\begin{align}
    \EX \big[ & (B_{\ell, \mathbf{y}} - B_{\ell-1, \mathbf{y}} - \beta_{S, \mathbf{x}(B_{\ell-1, \mathbf{y}})})^r \, \big| \, \mathbf{x}(B_{\ell, \mathbf{y}} + \beta_{S, \mathbf{x}(B_{\ell-1, \mathbf{y}})}) \big] \notag \\
    & \leq b^{(1)}_r \mathbf{x}(B_{\ell-1, \mathbf{y}} + \beta_{S, \mathbf{x}(B_{\ell-1, \mathbf{y}})})  + b^{(2)}_r \big( \mathbf{x}(B_{\ell-1, \mathbf{y}} + \beta_{S, \mathbf{x}(B_{\ell-1, \mathbf{y}})} \big)^2.  \label{eq::x_return_bound_1}
\end{align}

Under Assumption \ref{assumption::type_1_carrying_capacity}, there exists a constant $C$ such that $\mu_{i,j}(x_1) \leq C$ and $\chi_{i, j, k}(x_1) \leq C$ for all $x_1 \geq 0$ and $i, j, k \in \{1, \dots, d\}$, whereby
\begin{align*}
    \EX\big[ \mathbf{x} \big( B_{\ell-1, \mathbf{y}} + \beta_{i, \mathbf{x}(B_{\ell-1}, \mathbf{y})} \big) \boldsymbol{\xi} - \mathbf{x} \big( B_{\ell-1, \mathbf{y}} + \beta_{i-1, \mathbf{x}(B_{\ell-1}, \mathbf{y})} \big) \boldsymbol{\xi}  \big] & \leq \xi_\text{max} d C \\
    \EX\big[ \big\{ \mathbf{x} \big( B_{\ell-1, \mathbf{y}} + \beta_{i, \mathbf{x}(B_{\ell-1}, \mathbf{y})} \big) \boldsymbol{\xi} - \mathbf{x} \big( B_{\ell-1, \mathbf{y}} + \beta_{i-1, \mathbf{x}(B_{\ell-1}, \mathbf{y})} \big) \boldsymbol{\xi} \big\}^2 \big] & \leq (\xi_\text{max} d)^2 C.
\end{align*}

Using Jensen's inequality, it then follows that
\begin{align}
\begin{split}
    \EX \big[ \mathbf{x} \big( B_{\ell-1, \mathbf{y}} + \beta_{S, \mathbf{x}(B_{\ell-1}, \mathbf{y})} \big) \boldsymbol{\xi}  \, \big| \, \mathbf{x}(B_{\ell-1}, \mathbf{y}) \in \mathcal{C}^{z'}_y \big] & \leq \xi_\text{max} (z' + y + d C S) \\
    \EX \big[ \big( \mathbf{x} \big( B_{\ell-1, \mathbf{y}} + \beta_{S, \mathbf{x}(B_{\ell-1}, \mathbf{y})} \big) \boldsymbol{\xi} \big)^2 \, \big| \, \mathbf{x}(B_{\ell-1}, \mathbf{y}) \in \mathcal{C}^{z'}_y \big] & \leq (S+1)^2 \xi_\text{max}^2 \big[ (z' + y)^2 + d^2 C \big].
\end{split} \label{eq::x_return_bound_2}
\end{align}

Taking the expectation over Equation (\ref{eq::x_return_bound_1}) and substituting Equation (\ref{eq::x_return_bound_2}) yields the existence of a constant $c_r < \infty$ such that
\begin{align}
    \EX \big[ & (B_{\ell, \mathbf{y}} - B_{\ell-1, \mathbf{y}} - \beta_{S, \mathbf{x}(B_{\ell-1, \mathbf{y}})})^r \big] \leq c_r. \label{eq::bl_r_moment_bound}
\end{align}

Using Equations (\ref{eq::bs_r_moment_bound}) and (\ref{eq::bl_r_moment_bound}) and Jensen's inequality, for $r \in \mathbbm{N}$, we then bound
\begin{align}
    \EX[ (E_\mathbf{y})^r \, | \, U_\mathbf{y} = n] &  \leq (2n)^{r-1} \sum^n_{\ell = 1} \EX \big[ (B_{\ell, \mathbf{y}} - B_{\ell-1, \mathbf{y}} - \beta_{S, \mathbf{x}(B_{\ell-1, \mathbf{y}})})^r + (\beta_{S, \mathbf{x}(B_{\ell-1, \mathbf{y}})})^r \, \big| \, U_\mathbf{y} = n \big] \notag \\
    & \leq 2^{r-1} n^r \bigg( c_r + \frac{\Gamma(S+r)}{\Gamma(S) (\gamma_\text{min})^r} \bigg). \label{eq::conditional_extinction_bound}
\end{align}

Applying the law of total expectation, and using the conditional bound (\ref{eq::conditional_extinction_bound}) in addition to the geometric bound (\ref{eq::uy_geometric_bound}) on $U_\mathbf{y}$, we thus conclude that
\begin{align}
    \EX[ (E_\mathbf{y})^r ] & = \sum^\infty_{n=1} \EX[ (E_\mathbf{y})^r \, | \, U_\mathbf{y} = n] \mathbbm{P}( U_\mathbf{y} = n ) \notag \\
    & \leq r 2^{r-1} \bigg( c_r + \frac{\Gamma(S+r)}{\Gamma(S) (\gamma_\text{min})^r} \bigg) \sum^\infty_{n=1} n^{r-1} \mathbbm{P}( U_\mathbf{y} \geq n ) \notag \\
    & \leq r 2^{r-1} \bigg( c_r + \frac{\Gamma(S+r)}{\Gamma(S) (\gamma_\text{min})^r} \bigg) \sum^\infty_{n=1} n^{r-1} (1 - \upsilon)^\ell < \infty. \label{eq::ey_moment_bound}
\end{align}

For any $r \in \mathbbm{N}$, combining Lemma \ref{lemma::first_passage_time}, which bounds the $r^\text{th}$ moment for the first passage time to the finite domain $\mathcal{C}^{z'}_y$ given an initial condition $\mathbf{x}(0) \in (\mathbbm{Z}_{\geq 0})^d \setminus \mathcal{C}^{z'}_y$, with the bound (\ref{eq::ey_moment_bound}) on the $r^\text{th}$ moment for the time to extinction given an initial condition $\mathbf{x}(0) \in \mathcal{C}^{z'}_y$ allows us to conclude that the time to extinction given any initial condition $\mathbf{x}(0) \in (\mathbbm{Z}_{\geq 0})^d$ possesses a finite $r^\text{th}$ moment.
\end{subproof}

The stated result follows from Lemma \ref{lemma::extinction_moments}, which draws on Lemmas \ref{lemma::downcrossing} to \ref{lemma::first_passage_time}.

\end{proof}

\begin{remark} \label{remark::exponential_moment}
In Theorem \ref{theorem::extinction_almost_surely}, we proved that the time to extinction possesses finite moments of all orders $r \in \mathbbm{N}$. Since our upper bound for the $r^\text{th}$ moment for the first passage time to a finite domain $\mathcal{C} \supset \{ 0 \}$ (Lemma \ref{lemma::first_passage_time}) is asymptotically of order $O((r!)^2 c^r)$ for a constant $c>0$ as $r \to \infty$, our construction is not sufficient to establish the existence of a finite exponential moment for the time to extinction (in our setting, this is equivalent to the property of exponential ergodicity, as defined by \textcite{down1995exponential}).

In general, the existence of an exponential moment for the time to extinction is a necessary condition for the existence of a quasi-stationary distribution (Proposition 2.4 of \parencite{collet2012quasi}). 

In the special case where $\mathbf{x}$ is irreducible on the set of non-absorbing states $(\mathbbm{Z}^d) \setminus \{ \mathbf{0} \}$, Theorem 1.1 of \textcite{ferrari1995existence} implies that an exponential moment for the time to extinction is also sufficient for the existence of a quasi-stationary distribution. This is because the process $\mathbf{x}$ is conservative, stable and non-explosive, with the transition rates (\ref{eq::transition_rates}) yielding a unique $Q$-process, in addition to exhibiting `asymptotic remoteness' \parencite{ferrari1995existence}: the time to extinction $E_\mathbf{s}$ when the branching process is initialised in state $\mathbf{x}(0) = \mathbf{s}$ exhibits first order stochastic dominance over the maximum order statistic of $\lVert \mathbf{s} \rVert$ independent and identically distributed exponential random variables with mean $\gamma_\text{max}$, whereby
\begin{align*}
    \lim_{\lVert \mathbf{s} \rVert \to \infty} \mathbbm{P}( E_{\mathbf{s}} < t) \leq  \lim_{\lVert \mathbf{s} \rVert \to \infty} (1-e^{-\gamma_\text{max} t})^{\lVert \mathbf{s} \rVert} = 0 \text{ for any } t \geq 0.
\end{align*}

We conjecture that the time to extinction possesses an exponential moment, but have been unable to prove this conjecture thus far.
\end{remark}

\section{Scaling limits} \label{sec::scaling_limits}

In the single-type case, under sufficiently strict subcriticality conditions, a size-dependent branching processes initialised in the vicinity of its soft carrying capacity $K$ can be shown to linger in the domain $[(1-\varepsilon)K, (1+\varepsilon) K]$ for arbitrarily small $\varepsilon > 0$ for order $\mathcal{O}(\exp(c_\varepsilon K))$ time \parencite{jagers2011population, hamza2016establishment}. In the multitype case, the interplay between different types can yield more complex dynamical behaviour, including self-sustained oscillations; however, under some parameter regimes, we recover an analogous exponential time scale for persistence in the vicinity of an appropriate metastable state. We now draw on the classical limit theorems of \textcite{ethier2009markov} with a view towards characterising this behaviour.

In the ensuing analysis, we consider scaling limits with respect to the type 1 carrying capacity $K$. To this end, we construct sequences of density-dependent Markovian branching process models $\mathbf{x}^K$ indexed by $K$. We use the superscript $K$ only in the context of a sequence of processes indexed by $K$ and enforce Assumption \ref{assumption:density_dependent_process} below, motivated by the work of \parencite{ethier2009markov, prodhomme2023strong}. 

\begin{assumption} \label{assumption:density_dependent_process}
For the branching process $\mathbf{x}^K$ with type 1 carrying capacity $K$:
\begin{itemize}
    \item the death rate for a type $i$ individual when the type $1$ subpopulation is of size $K z_1$ takes the form $\gamma_i^K(K z_1) = g_i(z_1)$, where the function $g_i(z_1)$ with respect to the continuous variable $z_1$ is bounded, continuously differentiable with locally Lipschitz first derivative, and strictly positive on the domain $\mathbbm{R}_{\geq 0}$ with $\inf_{i \in \{1, \dots, d\}, z_1 \in \mathbbm{R}_{\geq 0}} g_i(z_1) > 0$.
    \item the probability that a type $i$ individual will give rise to an offspring batch of size $\boldsymbol{\ell} \in (\mathbbm{Z}_{\geq 0})^d$, when the type $1$ subpopulation is of size $K z_1$, takes the form $\pi^K_{(i, \boldsymbol{\ell})}(K z_1) = p_{(i, \boldsymbol{\ell})}(z_1)$, where each function $p_{(i, \boldsymbol{\ell})}(z_1)$ with respect to the continuous variable $z_1$ is continuously differentiable with locally Lipschitz first derivative, and either strictly positive (with $\inf_{z_1 \in \mathbbm{R}_{\geq 0}} p_{(i, \boldsymbol{\ell})}(z_1) > 0$) or identically zero, on the domain $\mathbbm{R}_{\geq 0}$. 
    \item the first $\mu^K_{i,j}(K z_1) = m_{i,j}(z_1)  = \sum_{\boldsymbol{\ell} \in \mathcal{J}} \ell_j p_{(i, \boldsymbol{\ell})}(z_1)$ and second $\chi^K_{i,j_1, j_2}(K z_1) = c_{i,j_1, j_2}(z_1) = \sum_{\boldsymbol{\ell} \in \mathcal{J}} \ell_{j_1} \ell_{j_2} p_{(i, \boldsymbol{\ell})}(z_1)$ moments of the offspring batch sizes are uniformly bounded, continuously differentiable functions of the continuous variable $z_1$ on the domain $\mathbbm{R}_{\geq 0}$.
\end{itemize}
\end{assumption}

Under Assumption \ref{assumption:density_dependent_process}, transition rates for the process $\mathbf{x}^K$ can be written in the density-dependent form
\begin{align*}
    \mathbf{x} \to \mathbf{x} + \boldsymbol{\ell} - \mathbf{e}_i \text{ at rate } K  \Big[ g_i \Big( \frac{x_1}{K} \Big) \cdot p_{(i, \boldsymbol{\ell})} \Big( \frac{x_1}{K} \Big) \cdot \frac{x_i}{K} \Big].
\end{align*}
Under Assumption \ref{assumption::type_1_carrying_capacity}, the process $\mathbf{x}^K$ is non-explosive; for any finite interval $[0, t]$, it therefore permits the representation
\begin{align*}
    \mathbf{x}^K(t) = \mathbf{x}^K(0) + \sum^d_{i=1} \sum_{\boldsymbol{\ell} \in \mathcal{J}} \big(\boldsymbol{\ell} - \mathbf{e_i} \big) Y_{(i, \boldsymbol{\ell})} \bigg( K \int^t_0 \bigg[ g_i \Big( \frac{x_1(s)}{K} \Big) \cdot p_{(i, \boldsymbol{\ell})} \Big( \frac{x_1(s)}{K} \Big) \cdot \frac{x_i(s)}{K} \bigg] ds \bigg)
\end{align*}
where $Y_{(i, \boldsymbol{\ell})}(\cdot)$ denote independent standard Poisson processes (Theorem 6.4.1 of \parencite{ethier2009markov}).

We introduce the following constants for the ensuing analysis:
\begin{align}
    g_\text{min} & := \inf_{i \in \{1, \dots, d\}, z \in \mathbbm{R}_{\geq 0}} g_i(z) \label{constant::scaling_gmin} \\
    g_\text{max} & := \sup_{i \in \{1, \dots, d\}, z \in \mathbbm{R}_{\geq 0}} g_i(z) \label{constant::scaling_gmax} \\
    b & := \sup_{i, j \in \{1, \dots, d\}, z \in \mathbbm{R}_{\geq 0}} \gamma_i(z) m_{i,j}(z) \label{constant::scaling_b}.
\end{align}

Under Assumption \ref{assumption:density_dependent_process}, it is necessarily the case that $g_\text{min} > 0$.

Hereafter, we denote by $A(z_1) \in \mathbbm{R}^{d \times d}$ the offspring generator matrix, with
\begin{align*}
     A_{i,j}(z_1) := g_j(z_1) \big[ m_{j,i}(z_1) - \delta_{i,j} \big].
\end{align*}
Set $\lambda_\text{max}(z_1)$ to be the spectral abscissa and $A'(z_1)$ to be the element-wise first derivative of $A(z_1)$. It follows from Assumption \ref{assumption::type_1_carrying_capacity} that:

\begin{assumptionp}{A$'$} \label{assumption::limit_theorem_scaled_carrying_capacity}
$\text{sign} (\lambda_\text{max}(z_1)) = \text{sign}(1 - z_1)$, with $A(z_1)$ converging element-wise in the limits $z_1 \to 0$ and $z_1 \to \infty$ respectively.
\end{assumptionp}

\subsection{Limit theorems} \label{sec::limit_theorems}

Let $\mathbf{X}(t)$ be the solution to the system of ordinary differential equations (ODEs)
\begin{align}
    \frac{d \mathbf{X}}{dt} &= A(X_1) \mathbf{X}\label{eq::FLLN_ODE}
\end{align}
and $\mathbf{V}(t)$ be the solution to the stochastic differential equation (SDE)
\begin{align}
    \mathbf{V}(t) = & \mathbf{V}(0) +  \sum^d_{i=1} \sum_{ \boldsymbol{\ell} \in \mathcal{J}} (\boldsymbol{\ell} - \mathbf{e}_i) W_{(i, \boldsymbol{\ell})} \bigg( \int^t_0 g_i \big(  X_1(s) \big) p_{(i, \boldsymbol{\ell})} \big(  X_1(s) \big) X_i(s) ds \bigg) \notag \\
    & + \int^t_0 \Big[ A \big( X_1(s) \big) + A' \big( X_1(s) \big) \cdot \mathbf{X}(s) \cdot (\mathbf{e}_1)^T \Big] \mathbf{V}(s) ds \label{eq::CLT_SDE}
\end{align}
where $ W_{(i, \boldsymbol{\ell})}$ represent independent (standard) Brownian motions for each $i \in \{1, \dots, d\}$ and $\boldsymbol{\ell} \in \mathcal{J}$, and the derivative $A'$ is calculated element-wise.

The functional law of large numbers (FLLN) of \textcite{ethier2009markov} yields strong convergence of the scaled sample paths $\mathbf{x}^K/K$ to the deterministic trajectory $\mathbf{X}$ on any finite time horizon.

\begin{theorem}[Functional law of large numbers, Theorem 11.2.1 of \parencite{ethier2009markov}] \label{theorem::FLLN}
Suppose Assumption \ref{assumption:density_dependent_process} holds and $\mathbf{X}(0) = \lim_{K \to \infty} \mathbf{x}^K(0)/K$. Then 
\begin{align*}
    \lim_{K \to \infty} \sup_{s \leq t} \bigg| \frac{\mathbf{x}^K(s)}{K} - \mathbf{X}(s) \bigg| = 0 \text{ almost surely for all } t \geq 0.
\end{align*}
\end{theorem}

The central limit theorem (CLT) of \textcite{ethier2009markov} yields weak convergence of the scaled difference $\sqrt{K} (\mathbf{x}^K/K - \mathbf{X})$ to the process $\mathbf{V}$.

\begin{theorem}[Central limit theorem, Theorem 11.2.3 of \parencite{ethier2009markov}] \label{theorem::central_limit}
Suppose Assumption \ref{assumption:density_dependent_process} holds and $\mathbf{V}(0) = \lim_{K \to \infty} \sqrt{K}(\mathbf{x}^K(0)/K - \mathbf{X}(0))$. Then
\begin{align*}
    \lim_{K \to \infty} \sqrt{K} \bigg( \frac{\mathbf{x}^K(t)}{K} - \mathbf{X}(t) \bigg) \to \mathbf{V}(t) \text{ weakly for any } t \geq 0.
\end{align*}
\end{theorem}

Under the stronger additional assumptions that offspring batch sizes have uniformly bounded support (Assumption \ref{assumption::bounded_batch_size}), and that the death rates $g(z_1)$ and mean offspring batch sizes $m_{i,j}(z_1)$ have uniformly bounded first and second derivatives, the asymptotic precision of the approximating process $\mathbf{X} + \mathbf{V}/\sqrt{K}$ relative to the scaled process $\mathbf{x}^K/K$ can be shown to be of order $\mathcal{O}(\log(K)/K)$ on any fixed time horizon (Theorem 4.4 of \textcite{kurtz1978strong}).

\subsection{Limiting behaviour of the FLLN} \label{sec::FLLN_limiting_behaviour}

Given an initial population size of order $\lVert \mathbf{x}^K(0) \rVert = \mathcal{O}(K)$, the FLLN and CLT of \parencite{ethier2009markov} yield approximating processes that are asymptotically valid in the limit $K \to \infty$ on any finite time horizon. Metastable behaviour, with a (sub)exponential (with respect to $K$) time scale for persistence in the vicinity of an appropriate non-zero state, is expected to emerge if the scaled process $\mathbf{x}^K/K$ reaches the domain of attraction of an exponentially stable equilibrium point \parencite{barbour1976quasi, barbour2012total, prodhomme2023strong} (when it exists) or in the vicinity of a stable limit cycle \parencite{bressloff2020phase} (when it exists) of the system of ODEs (\ref{eq::FLLN_ODE}) constituting the FLLN. We now address the limiting behaviour of the FLLN.

For a single type model with a soft carrying capacity ($d=1$), the FLLN can be shown to possess a unique non-zero equilibrium with domain of attraction $\mathbbm{R}_{>0}$ (since the population size is expected to decay above the soft carrying capacity, but grow below it), whereby metastable behaviour can be expected if a population size of order $\mathcal{O}(K)$ is attained prior to extinction. In contrast, under our multitype setting, with a soft carrying capacity imposed on the type 1 subpopulation, an exponentially stable equilibrium point of (\ref{eq::FLLN_ODE}) need not exist, let alone possess domain of attraction $(\mathbbm{R}_{\geq 0})^d \setminus \{ \mathbf{0} \}$; we can construct explicit examples of systems which exhibit periodic orbits (see Examples \ref{example::2D_limit_cycle} and \ref{example::3D_limit_cycle}, and Section \ref{sec::FLLN_periodic_2D_omega_limit} respectively).

In the ensuing analysis, we restrict our attention to the two cases detailed in Assumption \ref{assumption::irreducibility}.

\begin{assumption} \label{assumption::irreducibility}
Either:
\begin{enumerate}
    \item[(i)] $A(z_1)$ is irreducible for all $z_1 \geq 0$; or
    \item[(ii)] the principal submatrix $A_{2:d}(z_1)$ obtained by removing the first row and column of $A(z_1)$ is irreducible, while type 1 individuals generate no offspring for all $z_1 \geq 0$.
\end{enumerate}
\end{assumption}

The irreducible case is technically simplest. The second case in Assumption \ref{assumption::irreducibility} allows the interpretation of the type 1 subpopulation as a measure of cumulative exposure or `memory' to type $2, \dots, d$ individuals on a specified time scale, in line with the motivating framework of an immunity-modulated multi-stage parasitic disease (see Section \ref{sec::example_parasite_model} for details of this construction).

\subsubsection{Equilibrium solutions}

In Theorem \ref{theorem::FLLN_equilibrium} below, we characterise the equilibrium points of the system (\ref{eq::FLLN_ODE}) under Assumptions \ref{assumption::limit_theorem_scaled_carrying_capacity}, \ref{assumption::type_1_offspring}, \ref{assumption:density_dependent_process} and \ref{assumption::irreducibility}.

\begin{theorem} \label{theorem::FLLN_equilibrium}
Under Assumptions \ref{assumption::limit_theorem_scaled_carrying_capacity}, \ref{assumption::type_1_offspring}, \ref{assumption:density_dependent_process} and \ref{assumption::irreducibility}, the system of ODEs (\ref{eq::FLLN_ODE}) possesses bounded trajectories and is invariant on the non-negative orthant. There exist two equilibria in the non-negative orthant:
\begin{itemize}
    \item $\mathbf{X}^*=\mathbf{0}$, which is asymptotically unstable with no associated homoclinic orbits in the non-negative orthant.
    \item $\mathbf{X}^* = \mathbf{Y}^*$, where $\mathbf{Y}^*$ is the right eigenvector of $A(1)$ corresponding to the eigenvalue zero, normalised to yield $Y_1^*=1$. This equilibrium point is exponentially stable iff the Jacobian
    \begin{align}
        \boldsymbol{\theta} := A(1) +  A'(1) \cdot \mathbf{Y}^*  \cdot (\mathbf{e}_1)^T, \label{eq::theta_Jacobian}
    \end{align}
    is Hurwitz stable.
\end{itemize}
\end{theorem}
\begin{proof}

Observing that $\frac{dX_i}{dt}|_{X_i=0, X_{j \neq i} \geq 0} \geq 0$ for each $i=1, \dots, d$, we conclude that the system of ODEs (\ref{eq::FLLN_ODE}) is invariant on the non-negative orthant.

Under Assumption \ref{assumption::limit_theorem_scaled_carrying_capacity}, the inequality (\ref{eq::threshold_y}) yields the existence of a threshold $y$, positive vector $\boldsymbol{\xi} > 0 $ and constant $h > 0$ such that
\begin{align*}
    \boldsymbol{\xi}^T \frac{d \mathbf{X}}{dt} \leq -h  \boldsymbol{\xi}^T \mathbf{X} \text{ provided } X_1 > y.
\end{align*}

Suppose we initialise $\mathbf{X}(0) = \mathbf{Y}$ such that $Y_1 > y$. Set $T_\mathbf{Y} := \inf\{ t \geq 0: X_1(t) \leq y \, | \, \mathbf{X}(0) = \mathbf{Y})$. Then by a variation of Gronwall's inequality \parencite{pham2007variation},
\begin{align}
    \boldsymbol{\xi}^T \mathbf{X}(t) \leq e^{-h t} \boldsymbol{\xi}^T \mathbf{Y}\text{ for all } t \leq T_\mathbf{Y},  \label{eq::ode_bound_1}
\end{align}
whereby the system necessarily returns to the (sub)critical regime, that is,  $\liminf_{t \to \infty} X_1(t) \leq y$.

We now prove that the functional $\boldsymbol{\xi}^T \mathbf{X}$ remains bounded along trajectories. We begin by applying Gronwall's inequality to obtain the bounds
\begin{align}
    \frac{d X_i}{dt} \geq - g_\text{max} X_i & \implies X_i(t) \geq X_i(0) e^{-g_\text{max} t} \text{ for all } t \geq 0, i \in \{1, \dots, d\} \label{eq::ode_bound_2a} \\
    \boldsymbol{\xi}^T \frac{d \mathbf{X}}{dt} \leq b \boldsymbol{\xi}^T \mathbf{X} & \implies \boldsymbol{\xi}^T \mathbf{X}(t) \leq e^{b t} \boldsymbol{\xi}^T \mathbf{X}(0) \text{ for all } t \geq 0  \label{eq::ode_bound_2b}
\end{align}
where the constants $g_\text{max}$ and $b$ are given by Equations (\ref{constant::scaling_gmax}) and (\ref{constant::scaling_b}) respectively.

Under Assumption \ref{assumption::type_1_offspring}, there exists a constant $\varepsilon > 0$ such that
\begin{align}
    \frac{dX_1}{dt} \geq \varepsilon g_\text{min} \sum^d_{i=2} X_i - g_\text{max} X_1 \label{eq::ode_bound_3a}
\end{align}
where the constant $g_\text{min}$ is given by Equation (\ref{constant::scaling_gmin}). Using Equations (\ref{eq::ode_bound_2a}) and (\ref{eq::ode_bound_3a}), it then follows that
\begin{align}
    \frac{d}{dt} \bigg( X_1(t) e^{g_\text{max} t} \bigg) \geq \varepsilon g_\text{min} \bigg( \sum^d_{i=2} X_i(0) \bigg) & \implies X_1(t) \geq \varepsilon g_\text{min} \bigg( \sum^d_{i=2} X_i(0) \bigg) t e^{-g_\text{max} t} \notag \\
    & \implies X_1 \bigg( \frac{1}{g_\text{max}} \bigg) \geq \frac{\varepsilon g_\text{min}}{e g_\text{max}}  \bigg( \sum^d_{i=2} X_i(0) \bigg). \label{eq::ode_bound_3b}
\end{align}

By analogy with the proof of Theorem \ref{theorem::extinction_almost_surely}, we now consider the upcrossings and downcrossings of the deterministic process $\mathbf{X}$ over the type 1 threshold $X_1 = y$. For any constant $n > 1$, suppose we initialise $\mathbf{X}(0) = \mathbf{U}$ with $\sum^d_{i=2} U_i > \frac{n y e g_\text{max}}{\varepsilon g_\text{min}}$. Then Equations (\ref{eq::ode_bound_3b}) and (\ref{eq::ode_bound_2b}) imply that
\begin{align}
    X_1 \bigg( \frac{1}{g_\text{max}} \bigg) \geq ny \qquad \boldsymbol{\xi}^T \mathbf{X}(t) &\leq e^{\frac{b}{g_\text{max}}} \boldsymbol{\xi}^T \mathbf{U}  \text{ for all } t \leq \frac{1}{g_\text{max}}. \label{eq::ode_bound_4}
\end{align}

Using Equations (\ref{eq::ode_bound_1}) and (\ref{eq::ode_bound_2b}), we can then bound
\begin{align}
    T_\mathbf{\mathbf{X}( \frac{1}{g_\text{max}})} \geq \frac{\log(n)}{g_\text{max}} \implies \boldsymbol{\xi}^T \mathbf{X} \bigg( \frac{1}{g_\text{max}} + T_\mathbf{\mathbf{X}( \frac{1}{g_\text{max}})} \bigg) \leq e^{\frac{b}{g_\text{max}}} n^{-\frac{h}{g_\text{max}}} \boldsymbol{\xi}^T \mathbf{U}. \label{eq::ode_bound_5}
\end{align}

Now, choose $n$ sufficiently large such that $e^{\frac{b}{g_\text{max}}} n^{-\frac{h} {g_\text{max}}}  < \frac{1}{2}$. Then Equations (\ref{eq::ode_bound_4}) and (\ref{eq::ode_bound_5}) yield
\begin{align*}
    \boldsymbol{\xi}^T \mathbf{X} \bigg( \frac{1}{g_\text{max}} + T_\mathbf{\mathbf{X}( \frac{1}{g_\text{max}})} \bigg) < \frac{1}{2} \boldsymbol{\xi}^T \mathbf{U}  \text{ with } \boldsymbol{\xi}^T \mathbf{X}( t ) \leq e^{\frac{b \xi_\text{max}}{g_\text{max}}} \boldsymbol{\xi}^T \mathbf{U} \text{ for all } t \leq \frac{1}{g_\text{max}} + T_\mathbf{\mathbf{X}( \frac{1}{g_\text{max}})},
\end{align*}
that is, provided the total initial type $2, \dots, d$ subpopulation $\sum^d_{i=2} X_i(0)$ is sufficiently large, the functional $\boldsymbol{\xi}^T \mathbf{X}(\frac{1}{g_\text{max}} + T_\mathbf{\mathbf{X}( \frac{1}{g_\text{max}})})$ upon the first crossing of the threshold $X_1 = y$ after a fixed delay $\frac{1}{g_\text{max}}$, is strictly smaller than the initial condition $\mathbf{X}(0) \boldsymbol{\xi}$, with a uniform bound over the functional $\mathbf{X}(t) \boldsymbol{\xi} \leq e^{\frac{b \xi_\text{max}}{g_\text{max}}} \mathbf{X}(0) \boldsymbol{\xi}$ in the interval $t \in [0,  \frac{1}{g_\text{max}} + T_\mathbf{\mathbf{X}( \frac{1}{g_\text{max}})}]$. This allows us to conclude that $\boldsymbol{\xi}^T \mathbf{X}$ is necessarily bounded along trajectories of the system (\ref{eq::FLLN_ODE}), whereby the trajectories themselves are bounded.

Next, consider an element-wise non-negative equilibrium solution $\mathbf{X}^* \geq 0$ to the system of ODEs (\ref{eq::FLLN_ODE}). Then $A(X_1^*) \mathbf{X}^* = \boldsymbol{0}$, whereby either $\mathbf{X}^* = \mathbf{0}$ or $\mathbf{X}^*$ is a non-negative right eigenvector of the (singular) matrix $A(X_1^*)$ corresponding to the eigenvalue zero. We consider the two cases under Assumption \ref{assumption::irreducibility}:

\begin{enumerate}
    \item[Case (i)] Since $A(X_1)$ is irreducible for all $X_1 \geq 0$, by the Perron-Frobenius theorem (Theorem 1.5 of \parencite{seneta2006non}) and \parencite{athreya2004branching}, the right eigenvector $\boldsymbol{\nu}(X_1)$ corresponding to the (necessarily real) eigenvalue $\lambda_\text{max}(X_1)$ with largest real part is element-wise positive and the eigenspace associated with $\lambda_\text{max}(X_1)$ is one-dimensional. By the subinvariance theorem (Theorem 1.6 of \parencite{seneta2006non}), there exist no other non-negative right eigenvectors of $A(X_1)$ corresponding to real eigenvalues, barring scalar multiples of $\boldsymbol{\nu}(X_1)$. Under Assumption \ref{assumption::limit_theorem_scaled_carrying_capacity}, $\lambda_\text{max}(X_1)=0$ if and only if $X_1=1$. Consequently, there exists a unique non-zero equilibrium solution $\mathbf{Y}^*$ to the system of ODEs (\ref{eq::FLLN_ODE}), given by the right eigenvector of $A(1)$ corresponding to the eigenvalue zero, normalised to yield $Y_1^*=1$.
    \item[Case (ii)] Since the principal submatrix $A_{2:d}(X_1)$ is irreducible and $A_{2:d,1}(X_1)=0$, for all $X_1 \geq 0$, $\sigma(A(x_1)) = \sigma(A_{2:d}(X_1)) \cup \{A_{1,1}(X_1) \}$ where $\sigma(\cdot)$ denotes the spectrum and $A_{1,1}(X_1) = -g_1(X_1)  < 0$ by assumption. To obtain a non-zero equilibrium solution $\mathbf{Y}^*$ in the non-negative orthant, $A_{2:d}(X_1)$ must possess a non-negative right eigenvector $\boldsymbol{\nu}(X_1)$ corresponding to the eigenvalue zero; by the Perron-Frobenius theorem and subinvariance theorem (Theorems 1.5 and 1.6 of \parencite{seneta2006non}), \parencite{athreya2004branching} and Assumption \ref{assumption::limit_theorem_scaled_carrying_capacity}, this can occur only if the (necessarily real and simple) eigenvalue of $A_{2:d}(X_1)$ with maximal real part is zero, whereby $X_1=1$. We are guaranteed $\boldsymbol{\nu}(1)$ is element-wise positive, while $A_{1,2:d}(1) > 0$ under Assumption \ref{assumption::type_1_offspring}. The unique non-zero equilibrium solution thus takes the form $\mathbf{Y}^* = (1 , 
    g_1(1) \boldsymbol{\nu}(1)/(A_{1,2:d}(1) \cdot \boldsymbol{\nu}(1) )  )$; this is likewise the right eigenvector of $A(1)$ corresponding to the eigenvalue zero, normalised to yield $Y_1^*=1$. 
\end{enumerate}

The Jacobian of the system, evaluated at the equilibrium point $\mathbf{X}^*$, takes the form
\begin{align*}
    J( \mathbf{X}^*) = A(X_1^*) +  A'(X_1^*) \cdot \mathbf{X}^* \cdot (\mathbf{e}_1)^T;
\end{align*}
$\mathbf{X}^*$ is exponentially stable iff $J(\mathbf{X}^*)$ is Hurwitz stable. 

Under Assumption \ref{assumption::limit_theorem_scaled_carrying_capacity}, $J( \mathbf{0}) = A(0)$ has a strictly positive spectral abscissa, whereby, $\mathbf{X}^* = \mathbf{0}$ is asymptotically unstable. To show that there are no homoclinic orbits in the non-negative orthant associated with this fixed point, we consider the two cases under Assumption \ref{assumption::irreducibility}:
\begin{enumerate}
    \item[Case (i)] By the Perron-Frobenius theorem (Theorem 1.5 of \parencite{seneta2006non}), \parencite{athreya2004branching} and Assumption \ref{assumption::limit_theorem_scaled_carrying_capacity}, there exists $\delta > 0$ such that
    \begin{align*}
        \boldsymbol{\zeta}^T A(x_1) \geq \frac{\lambda_\text{max}(0)}{2} \boldsymbol{\zeta} \text{ for all } 0 \leq x_1 \leq \delta
    \end{align*}
    where $\boldsymbol{\zeta} > 0$ is the left eigenvector of the (irreducible) matrix $A(0)$ corresponding to the dominant eigenvalue $\lambda_\text{max}(0) > 0$. Therefore,
    \begin{align}
        \frac{d (\boldsymbol{\zeta}^T \mathbf{X})}{dt} \geq \frac{\lambda_\text{max}(0)}{2} (\boldsymbol{\zeta}^T \mathbf{X}) \text{ provided } 0 \leq X_1 \leq \delta. \label{eq::ode_bound_delta}
    \end{align}
    Now, suppose a point $\mathbf{Z} \in (\mathbbm{R}_{\geq 0})^d$ lies on the stable manifold for $\mathbf{0}$. Suppose we initialise $\mathbf{X}(0) = \mathbf{Z}$. Then there exists $W \geq 0$ such that $0 \leq X_1(t) \leq \delta$ for all $t \geq W$. By Equation (\ref{eq::ode_bound_2a}), $\boldsymbol{\zeta}^T \mathbf{X}(W) \geq e^{-g_\text{max} W} \boldsymbol{\zeta}^T \mathbf{Z}$. Then Gronwall's inequality and Equation (\ref{eq::ode_bound_delta}) yield the bound
    \begin{align*}
        \boldsymbol{\zeta}^T \mathbf{X}(W + t) \geq e^{\frac{\lambda_\text{max}(0)}{2} t - g_\text{max} W} \mathbf{Z} \to \infty \text{ in the limit } t \to \infty,
    \end{align*}
    which is a contradiction. Therefore, the intersection of the stable manifold of $\mathbf{0}$ with the non-negative orthant is empty. This excludes the existence of homoclinic orbits corresponding to $\mathbf{0}$ in the non-negative orthant.
    \item[Case (ii)] It is straightforward to show that the stable manifold for $\mathbf{0}$ includes the line $S_1 := \mathbbm{R}_{\geq 0} \times \{ 0\}^{d-1}$. To show that the intersection of the stable manifold with the non-negative orthant is restricted to $S_1$, we follow analogous reasoning to Case (i). Using the Perron-Frobenius theorem (Theorem 1.5 of \parencite{seneta2006non}), \parencite{athreya2004branching} and Assumption \ref{assumption::limit_theorem_scaled_carrying_capacity}, there exists $\delta' > 0$ such that
    \begin{align*}
        (\boldsymbol{\zeta}_{2:d})^T A_{2:d}(x_1) \geq \frac{\lambda_\text{max}(0)}{2} \boldsymbol{\zeta}_{2:d} \text{ for all } 0 \leq x_1 \leq \delta
    \end{align*}
    where $\boldsymbol{\zeta}_{2:d} > 0$ is the left eigenvector of the (irreducible) matrix $A_{2:d}(0)$ corresponding to the dominant eigenvalue $\lambda_\text{max}(0) > 0$. If a point $\mathbf{Z} \in (\mathbbm{R}_{\geq 0})^d \setminus S_1$ lies on the stable manifold for $\mathbf{0}$, then an identical argument yields the contradiction  $(\boldsymbol{\zeta}_{2:d})^T \mathbf{X}_{2:d}(t) \to \infty$ in the limit $t \to \infty$. The intersection of the stable manifold of $\mathbf{0}$ with the non-negative orthant is thus restricted to $S_1$. Provided $\mathbf{X}(0) \in S_1$, we can readily verify that $\lim_{t \to -\infty} \mathbf{X}_1(t) \to \infty$, whereby $S_1$ does not intersect with the unstable manifold of $\mathbf{0}$; this then precludes the existence of a homoclinic orbit associated with $\mathbf{0}$ in the non-negative orthant.
\end{enumerate}
\end{proof}

Given $\mathbf{X}(0) = \lim_{K \to \infty} \mathbf{x}^K(0)/K = \mathbf{Y}^*$, the limiting process $\mathbf{V}$ (Equation (\ref{eq::CLT_SDE})) derived under the CLT \parencite{ethier2009markov} reduces to an Ornstein-Uhlenbeck process, with constant drift matrix $\boldsymbol{\theta}$ (Equation (\ref{eq::theta_Jacobian})) and local covariance matrix
\begin{align}
    \boldsymbol{\Upsilon} = \sum^d_{j=1} g_j(1) Y^*_j  \mathbf{C}_j(1)  + G_{\mathbf{Y}^*} \mathbf{Y}^* \label{eq::S_local_covariance}
\end{align}
where $\mathbf{C}_j(1) = ( c_{j, i_1, i_2}(1))_{1 \leq i_1, i_d \leq d}$ and $G_{\mathbf{Y}^*} = \text{diag}\{ g_1(1) Y_1^*, \dots, g_d(1) Y_d^* \}$. By \parencite{gardiner1985handbook}, if $\boldsymbol{\theta}$ is Hurwitz stable (whereby the fixed point $\mathbf{Y}^*$ is exponentially stable), this Ornstein-Uhlenbeck process is mean-reverting, with stationary solution 
\begin{align}
    \mathbf{V}^* \sim \text{Normal}(\mathbf{0}, \mathbf{\Sigma}) \text{ where } \mathbf{\Sigma} \text{ uniquely solves } \boldsymbol{\theta} \mathbf{\Sigma} + \mathbf{\Sigma} \boldsymbol{\theta}^T = -\boldsymbol{\Upsilon}. \label{eq::sigma_lyapunov}
\end{align}

\subsubsection{Asymptotic stability}

We now address the asymptotic stability of the non-zero equilibrium $\mathbf{Y}^*$ of the system of ODEs (\ref{eq::FLLN_ODE}). We begin by observing that $-A(1)$ is necessarily a singular $M$-matrix (under Assumption \ref{assumption::limit_theorem_scaled_carrying_capacity}, $A(1)$ has non-negative off-diagonal entries and spectral abscissa zero; under the irreducibility Assumption \ref{assumption::irreducibility}, we can apply the Perron-Frobenius theorem \parencite{athreya2004branching, seneta2006non} to show that $A(1)$ is necessarily singular if it has spectral abscissa zero, and then that zero is an algebraically simple eigenvalue of $A(1)$). Therefore, under Assumptions \ref{assumption::type_1_carrying_capacity} and \ref{assumption::irreducibility}, the negated Jacobian $-\boldsymbol{\theta}$ evaluated at $\mathbf{Y}^*$ (Equation (\ref{eq::theta_Jacobian})) constitutes a rank 1 perturbation of a singular $M$-matrix  \parencite{bierkens2014singular}. If $A(x)$ is component-wise analytic in some neighbourhood of $x=1$ and the dominant eigenvalue $\lambda_\text{max}(x)$ of $A(x)$ is strictly decreasing at $x=1$, that is, $\frac{d \lambda_\text{max}}{dx}|_{x=1} < 0$, then we can apply Theorem 5 of \parencite{lancaster1964eigenvalues} and Lemma 1.4 of \parencite{bierkens2014singular} to conclude that $\boldsymbol{\theta}$ is non-singular.

As a special case, suppose that $\frac{d \lambda_\text{max}}{dx}|_{x=1} < 0$ and the rank 1 perturbation $A'(1) \mathbf{Y}^* \leq 0$ (that is, the rate of production of each type, evaluated at the non-zero equilibrium $\mathbf{Y}^*$, has a non-positive derivative with respect to the type 1 subpopulation size). Then Lemmas 1.4, 2.10 and 2.18 of \parencite{bierkens2014singular} tell us that $-\boldsymbol{\theta}$ is a $P_0$ matrix (that is, possesses non-negative principal minors), and that all \textit{real} eigenvalues of the Jacobian $\boldsymbol{\theta}$ are strictly negative. In the case $d=2$, direct application of the Routh-Hurwitz criterion (as in Theorem 2.7(i) of \parencite{bierkens2014singular}) yields that this is a sufficient condition for asymptotic stability of the non-zero fixed point $\mathbf{Y}^*$. However, asymptotic stability cannot be established in generality for the case $d \geq 3$ without additional assumptions on the structure of the Jacobian $\boldsymbol{\theta}$ (such as those in Theorem 2.7 of \parencite{bierkens2014singular}).

\subsubsection{On the planar case $d=2$}

In the planar case $d=2$, we can apply Theorem \ref{theorem::FLLN_equilibrium} to characterise the limiting behaviour of the system of ODEs (\ref{eq::FLLN_ODE}) based on the classification of the (unique) non-zero equilibrium point $\mathbf{Y}^*$. As per Definition 8.1.1 of \parencite{wiggins2003introduction}, we define the $\omega$-limit set $\omega(\mathbf{X}_0)$ associated with $\mathbf{X}_0 \in (\mathbbm{R}_{\geq 0})^d$ as follows: $\mathbf{Z} \in \omega(\mathbf{X}_0)$ if, provided $\mathbf{X}(0) = \mathbf{X}_0$, there exists a sequence $\{ t_n \}_{n \in \mathbbm{N}}$ with $t_n \to \infty$ such that $\mathbf{X}(t_n) \to \mathbf{Z}$ in the limit $n \to \infty$.

\begin{theorem} \label{theorem::planar_FLLN_limiting_behaviour}
Suppose $d=2$ and Assumptions \ref{assumption::limit_theorem_scaled_carrying_capacity}, \ref{assumption::type_1_offspring}, \ref{assumption:density_dependent_process} and \ref{assumption::irreducibility} hold. Then for the system of ODEs (\ref{eq::FLLN_ODE}), for any $\mathbf{X}_0 \in \mathbbm{R}_{\geq 0} \times \mathbbm{R}_{>0}$:
\begin{itemize}
    \item If the (unique) fixed point $\mathbf{Y}^* > 0$ is a sink, source or centre, then the $\omega$-limit set $\omega(\mathbf{X}_0)$ is either $\{ \mathbf{Y}^* \}$ or a periodic orbit enclosing $\mathbf{Y}^*$.
    \item If the (unique) fixed point $\mathbf{Y}^* > 0$ is a hyperbolic saddle, then the $\omega$-limit set $\omega(\mathbf{X}_0)$ is either $\{ \mathbf{Y}^* \}$ or the union of $\{ \mathbf{Y}^* \}$ and associated homoclinic orbits.
\end{itemize}
\end{theorem}

\begin{proof}
By Theorem \ref{theorem::FLLN_equilibrium}, the system (\ref{eq::FLLN_ODE}) possesses bounded trajectories and is invariant on the non-negative orthant, with a single fixed point $\mathbf{0}$ on the boundary and a unique fixed point $\mathbf{Y}^*$ in the interior of the non-negative orthant. Further, the intersection of the stable manifold of $\mathbf{0}$ with the non-negative orthant is restricted to the line $\mathbbm{R}_{\geq 0} \times \{ 0 \}$. Using a stronger form of the Poincare-Benidixson theorem (Theorem 2.7 of \textcite{ramazi2024tightening}), we may thus conclude that  that the $\omega$-limit set for any point $\mathbf{X}_0 \in \mathbbm{R}_{\geq 0} \times \mathbbm{R}_{> 0}$ must either be $\{ \mathbf{Y}^* \}$; a periodic orbit enclosing $\mathbf{Y}^*$; or the union of $\{ \mathbf{Y}^* \}$ and associated homoclinic orbits. By index theory (Corollary 6.0.2 of \parencite{wiggins2003introduction}), a periodic orbit can only enclose the (unique) interior fixed point $\mathbf{Y}^*$ if it is a sink, source or centre; periodic orbits are thus disallowed if $\mathbf{Y}^*$ is a hyperbolic saddle. In contrast, homoclinic orbits corresponding to $\mathbf{Y}^*$ can exist only if $\mathbf{Y}^*$ is a hyperbolic saddle: by the stable manifold theorem (Theorem 3.2.1 of \parencite{wiggins2003introduction}), if $\mathbf{Y}^*$ is a sink (source), then there is no corresponding unstable (stable) manifold; if $\mathbf{Y}^*$ is a center, then it possesses neither a stable nor unstable manifold.
\end{proof}

From Theorem \ref{theorem::planar_FLLN_limiting_behaviour}, it is immediate that a sufficient condition for the system  (\ref{eq::FLLN_ODE}) to exhibit periodic orbits in the case $d=2$ is that the (unique) fixed point $\mathbf{Y}^* > 0$ is a source: for any $\mathbf{X}_0 \in (\mathbbm{R}_{\geq 0} \times \mathbbm{R}_{>0}) \setminus \{ \mathbf{Y}^* \}$, the corresponding omega limit set $\omega(\mathbf{X}_0)$ is necessarily a periodic orbit enclosing $\mathbf{Y}^*$. In Example \ref{example::2D_limit_cycle}, we show via subcritical Hopf bifurcation that periodic orbits may still exist when the (unique) fixed point $\mathbf{Y}^* > 0$ is a sink. 

Under the assumption $\frac{d\lambda_\text{max}}{dx}|_{x=1} < 0$ (with $A(x)$ component-wise analytic in some neighbourhood of $x=1$) and the element-wise monotonicity condition $A'(x) \leq 0$ --- which includes the simplified setting of fixed death rates $g_i$ and mean offspring batch sizes $m_{i,j}(x_1)$ that decrease monotonically as a function of the scaled type 1 subpopulation size $x_1$ --- we can apply Dulac's criterion (Theorem 4.1.2 of \parencite{wiggins2003introduction}) to show that the non-zero fixed point $\mathbf{Y}^*$ is a sink with domain of attraction $\mathcal{D} \supseteq \mathbbm{R} \times \mathbbm{R}_{>0}$ (Corollary \ref{corollary::FLLN_limiting_behaviour_simple}); in this simplified setting, we can thus expect metastable behaviour to emerge if a type 2 subpopulation size of order $\mathcal{O}(K)$ is attained prior to extinction, in direct analogy to the single-type case.

\begin{corollary} \label{corollary::FLLN_limiting_behaviour_simple}
Suppose $d=2$ and Assumptions \ref{assumption::limit_theorem_scaled_carrying_capacity}, \ref{assumption::type_1_offspring}, \ref{assumption:density_dependent_process} and \ref{assumption::irreducibility} hold. If $\frac{d \lambda_\text{max}}{dx}|_{x=1} < 0$ (with $A(x)$ component-wise analytic in some neighbourhood of $x=1$) and $A'(x) \leq 0$ element-wise for all $x \geq 0$, then the (unique) fixed point $\mathbf{Y}^* > 0$ of the system (\ref{eq::FLLN_ODE}) is asymptotically stable with domain of attraction $\mathcal{D} \supseteq \mathbbm{R}_{\geq 0} \times \mathbbm{R}_{>0}$.
\end{corollary}

\begin{proof}
Since $\frac{d \lambda_\text{max}}{dx}|_{x=1} < 0$ and $A'(1) \mathbf{Y}^* \leq 0$ by assumption, we can apply Theorem 5 of \parencite{lancaster1964eigenvalues} and Theorem 2.7(i) of \parencite{bierkens2014singular} to show that the unique non-zero equilibrium point $\mathbf{Y}^*$ is an asymptotically stable sink. Under Assumption \ref{assumption::type_1_offspring} and the element-wise monotonicity condition $A'(x) \leq 0$, we can also bound
\begin{align*}
    \nabla \cdot \bigg( \frac{A(X_1) \mathbf{X}}{X_1 X_2} \bigg) = & \frac{A_{11}'(X_1)}{X_2} + \frac{A_{12}'(X_1)}{X_1} - \frac{A_{12}(X_1)}{X_1^2} - \frac{A_{21}(X_1)}{X_2^2}   \leq -\frac{\varepsilon g_\text{min}}{X_1^2} < 0
\end{align*}
where the constant $g_\text{min}$ is given by Equation (\ref{constant::scaling_gmin}). By Dulac's criterion (Theorem 4.1.2 of \parencite{wiggins2003introduction}), there are no closed orbits within the domain $(\mathbbm{R}_{> 0})^2$. The claim follows from Theorem \ref{theorem::planar_FLLN_limiting_behaviour}.
\end{proof}

\subsubsection{On the case $d \geq 3$}

In the case $d \geq 3$, we can show by counter-example that the element-wise monotonicity condition $A'(x) \leq 0$ is neither sufficient to establish asymptotic stability of the fixed point $\mathbf{Y}^* > 0$, nor rule out periodic orbits.

To this end, consider an analytic family of vector fields of the form $A_{\alpha}(x) = \alpha F(x) + A$, such that $F(1) = 0$, satisfying Assumptions \ref{assumption::type_1_carrying_capacity}, \ref{assumption::type_1_offspring} and \ref{assumption::irreducibility}. Then the family of systems  $\frac{d\mathbf{X}}{dt} = A_\alpha(X_1) \mathbf{X}$ shares a unique non-zero equilibrium point $\mathbf{Y}^*$ in the non-negative orthant, with corresponding Jacobian $\boldsymbol{\theta}_\alpha = A + \alpha F'(1) \cdot \mathbf{Y}^* \cdot (\mathbf{e}_1)^T$. Given the dominant eigenvalue of $A_\alpha(x)$ is strictly decreasing at $x=1$ and $F'(1) \leq 0$ element-wise, we can conclude that $\mathbf{Y}^*$ is asymptotically stable (equivalently, $\boldsymbol{\theta}_\alpha$ is Hurwitz stable) for sufficiently small $\alpha$: Lemma 2.11 of \parencite{bierkens2014singular} yields the existence of a constant $\alpha_0 > 0$ such that $\mathbf{Y}^*$ is asymptotically stable for all $0 < \alpha \leq \alpha_0$. However, if there exists $\alpha_1>0$ such that $\mathbf{Y}^*$ is asymptotically unstable for $\alpha= \alpha_1$, then Theorem 5.1 of \parencite{vassena2025mass} yields the existence of a constant $\alpha^*$ such that the system $\frac{d \mathbf{X}}{dt} = A_{\alpha^*}(X_1) \mathbf{X}$ exhibits periodic orbits. A system that appears to exhibit a stable limit cycle in the case $d=3$, with fixed individual death rates and monotonicity in mean offspring batch sizes, is presented numerically in Example \ref{example::3D_limit_cycle}.

\section{On the emergence of metastable behaviour} \label{sec::metastable_behaviour}

We are principally interested in two metastable behaviours that can emerge under our notion of a type 1 carrying capacity: Gaussian fluctuation around a non-zero state, which is expected to occur once the process reaches the domain of attraction of an exponentially stable non-zero equilibrium of the FLLN (\ref{eq::FLLN_ODE}) (if it exists) \parencite{barbour1976quasi, barbour2012total, prodhomme2023strong}; or persistence in the vicinity of a stable limit cycle of the FLLN (\ref{eq::FLLN_ODE}) (when it exists) \parencite{bressloff2020phase}. Hereafter, we additionally enforce the assumption that the offspring batch sizes have uniformly bounded support (Assumption \ref{assumption::bounded_batch_size}).

\subsection{Metastability in the vicinity of an exponentially stable equilibrium} \label{sec::metastability_sink}

Suppose that the non-zero equilibrium $\mathbf{Y}^*$ of the FLLN (\ref{eq::FLLN_ODE}) is asymptotically stable. We would then expect that $\mathbf{x}^K(t)/K$ will persist in the vicinity of $\mathbf{Y}^*$ over an exponential time scale: given $\mathbf{x}^K(0)/K = \mathbf{Y}^*$, Proposition 3.7 of \textcite{prodhomme2023strong} stipulates that large deviations of order $\mathcal{O}(1)$ between the scaled process $\mathbf{x}^K/K$ and the fixed point $\mathbf{Y}^*$ are expected to occur on time scales of order $\mathcal{O}(\exp(\beta' K))$ (c.f. Theorem 6 of \parencite{jagers2011population} in the single-type case with a soft carrying capacity).

More precise results can also be attained. Denote by $\mathcal{U}$ the domain of attraction of $\mathbf{Y}^*$ for the system (\ref{eq::FLLN_ODE}). Given $\mathbf{x}^K(0)/K \in \mathcal{U}$, Theorem 1.1 of \textcite{prodhomme2023strong} tells us that the Gaussian approximation $\mathbf{X} + \mathbf{V}/\sqrt{K}$ to the scaled process $\mathbf{x}/K$ retains precision $\alpha \log(K)/K$ on polynomial time scales $\mathcal{O}(K^{c_\alpha})$ for sufficiently large $\alpha$; further, given a function $\Delta: \mathbbm{R}_{>0} \to \mathbbm{R}_{>0}$ such that $K^{-\frac{1}{2}} \leq \Delta(K) \ll 1$,  moderate deviations of order $\Delta(K)$ between the scaled process $\mathbf{x}^K/K$ and the approximating process $\mathbf{X} + \mathbf{V}/ \sqrt{K}$ are expected to emerge on subexponential time scales of order $\mathcal{O}(\exp(\beta K \Delta(K)))$. In contrast, given an arbitrary initial condition $\mathbf{x}^K(0)/K$ and several additional technical conditions, Theorem 4.4 of \textcite{kurtz1978strong} yields the weaker result that the approximation $\mathbf{X} + \mathbf{V}/\sqrt{K}$ maintains precision $\alpha \log(K)/K$ (for sufficiently large $\alpha$) on fixed time scales only.

\begin{theorem}[Based on Theorem 1.1 of \textcite{prodhomme2023strong}] \label{theorem::moderate_deviations}
Suppose that Assumptions \ref{assumption::bounded_batch_size} and \ref{assumption:density_dependent_process} hold, and that the Jacobian $\boldsymbol{\theta}$ (Equation (\ref{eq::theta_Jacobian})) evaluated at the fixed point $\mathbf{Y}^*$ is Hurwitz stable. Let $\mathcal{D} \subseteq \mathcal{U}$ be a compact subset of the domain of attraction of $\mathbf{Y}^*$ for the system (\ref{eq::FLLN_ODE}). There exist constants $C, \beta, \alpha > 0$ such that, for any initial condition of the form $\lim_{K \to \infty} \mathbf{x}^K(0)/K = \mathbf{X}(0) \in \mathcal{D}$ and for every $\Delta: \mathbbm{R}_{>0} \to \mathbbm{R}_{>0}$ satisfying $\alpha \log(K)/K \leq \Delta(K) \ll 1$, we can construct a coupling under which, for sufficiently large $K$,
\begin{align*}
    \mathbbm{P} \Bigg( \sup_{0 \leq t \leq T} \bigg\lVert \frac{\mathbf{x}^K(t)}{K} - \mathbf{X}(t) - \frac{\mathbf{V}(t)}{\sqrt{K}} \bigg\rVert_2 > \Delta(K) \Bigg) \leq C (T+1) \exp ( -\beta K \Delta(K)) \text{ for all } T \geq 0.
\end{align*}
\end{theorem}

Denote by $\mathbf{V}^*$ a stationary Ornstein-Uhlenbeck process with drift matrix $\boldsymbol{\theta}$ (Equation (\ref{eq::theta_Jacobian})) and local covariance matrix $\boldsymbol{\Upsilon}$ (Equation (\ref{eq::S_local_covariance})). Suppose that $\boldsymbol{\theta}$ (Equation (\ref{eq::theta_Jacobian})) is Hurwitz stable with spectral abscissa $-\lambda^* < 0$. Provided $\mathbf{x}^K(0)/K \in \mathcal{U}$ and following a ``transitory period'' of length $(6/\lambda^*) \log(K)$, Corollary 1.3 of \textcite{prodhomme2023strong} yields the stronger result that the scaled process $\mathbf{x}^K/K$ is well-approximated by the process $\mathbf{Y}^* + \mathbf{V}^*/\sqrt{K}$, with precision $\alpha \log(K)/K \leq \Delta(K) \ll 1$ on time scales of order $\mathcal{O}(\exp(\beta^* K \Delta(K)))$ for some sufficiently large constant $\alpha$.

\subsection{Metastability in the vicinity of a stable limit cycle} \label{sec::metastability_limit_cycle}

Now, suppose that the system (\ref{eq::FLLN_ODE}) exhibits a stable limit cycle $\bar{\mathbf{Y}}$ with natural frequency $\omega_0$, whereby $\bar{\mathbf{Y}}(t + 2 \pi / \omega_0) = \bar{\mathbf{Y}}(t)$; in the case $d=2$, this limit cycle would necessarily enclose the fixed point $\mathbf{Y}^* > 0$ (Theorem \ref{theorem::planar_FLLN_limiting_behaviour}). Given the scaled process $\mathbf{x}^K/K$ is initialised in a sufficiently small neighbourbood of $\bar{\mathbf{Y}}(t)$, we use the work of \textcite{bressloff2020phase} to establish the time scale for persistence of the scaled process $\mathbf{x}^K/K$ in the vicinity of $\bar{\mathbf{Y}}$. While the framework of \textcite{bressloff2020phase} is motivated in the context of stochastic chemical reaction networks with mass action kinetics and polynomial transition rates, to the best of our understanding, the underlying reasoning generalises to our system under some additional technical assumptions.

Following \parencite{bressloff2020phase}, in the vicinity of the limit cycle $\bar{\mathbf{Y}}(t)$, we assume that the deterministic system (\ref{eq::FLLN_ODE}) permits a Floquet decomposition of the form
\begin{align}
    \mathbf{Z}(t) = \sum^d_{i=1} a_i \mathbf{p}_i(t) e^{\phi_i t} \label{eq::Floquet_decomposition}
\end{align}
where $\phi_1, \dots, \phi_d$ are Floquet exponents ordered such that $\text{Re}(\phi_1) \leq \dots \leq \text{Re}(\phi_d)$;  $a_1, \dots, a_d$ are constant coefficients; and $\mathbf{p}_1(t), \dots, \mathbf{p}_d(t)$ are the columns of the (unique) $2 \pi$-periodic matrix $P(\omega_0 t)$ such that
\begin{align*}
    P(\psi)^{-1} \mathbf{\bar{Y}}'(\psi) = \mathbf{1}^T \mathbf{e}_1.
\end{align*}
Provided the limit cycle is asymptotically stable, it is necessarily the case that $\phi_1 = 0$ (because of ``phase shift invariance'' along the limit cycle), while $\text{Re}(\phi_2) < 0$ \parencite{bressloff2020phase}.

To characterise the amplitude of the variation between the scaled process $\mathbf{x}^K/K(t)$ and the limit cycle $\bar{\mathbf{Y}}$, \textcite{bressloff2020phase} introduce a variational principle to formulate a stochastic phase $\varsigma(t)$ for the scaled process $\mathbf{x}^K/K(t)$. Setting
\begin{align*}
    \mathcal{G}_0(\mathbf{y}, z, \psi) := -2 \Big\langle P^{-1} ( \psi ) \big( \mathbf{y} - \mathbf{Y}(z) \big), \, P^{-1} ( \psi )  \mathbf{Y}'(z) \big)  \Big\rangle,
\end{align*}
they define the stochastic phase $\varsigma(t)$ to be the solution of the implicit equation
\begin{align}
    \mathcal{G}_0 \big( \mathbf{x}^K(t)/K, \varsigma(t), \varsigma(t) \big) = 0 \label{eq::stochastic_phase_equation}
\end{align}

We now state the result of \textcite{bressloff2020phase} in our setting.

\begin{theorem}[Based on \textcite{bressloff2020phase}]\label{theorem::limit_cycle_metastability}
Suppose that Assumptions \ref{assumption::bounded_batch_size} and \ref{assumption:density_dependent_process} hold, and that the functions $g_i: \mathbbm{R}_{\geq 0} \to \mathbbm{R}_{\geq 0}$ and $m_{i,j}: \mathbbm{R}_{\geq 0} \to \mathbbm{R}_{\geq 0}$ are twice continuously differentiable with uniformly bounded second derivatives. Further suppose that the system (\ref{eq::FLLN_ODE}) possesses an asymptotically stable limit cycle $\bar{\mathbf{Y}}(t)$ which permits a Floquet decomposition of the form (\ref{eq::Floquet_decomposition}). Take $K$ to be sufficiently large such that $\sup_{\tau \in [0, 2 \pi/\omega_0]} \max \{ \lVert P(\tau) \rVert_F, \lVert P(\tau)^{-1} \rVert)_F \} = \mathcal{O}(1)$ (where $\lVert \cdot \rVert_F$ denotes the Frobenius norm) and set $\delta =-\text{Re}(\phi_2) > 0$. Define the weighted amplitude of deviation from the limit cycle
\begin{align*}
    \mathbf{w}(t) := P \big( \varsigma(t) \big)^{-1} \bigg[ \frac{\mathbf{x}^K(t)}{K} - \bar{\mathbf{Y}} \big( \varsigma(t) \big) \bigg]
\end{align*}
where the stochastic phase $\varsigma(t)$ solves Equation (\ref{eq::stochastic_phase_equation}). Then for sufficiently small $\Delta(K) \leq \delta/4$, given $\lVert \mathbf{w}(0) \rVert_2 \leq \Delta(K)/8$, there exists a constant $\beta^* > 0$ such that
\begin{align*}
    \mathbbm{P} \bigg( \sup_{0 \leq t \leq T} \lVert \mathbf{w}(t) \rVert_2 \geq \Delta(K)  \bigg) \leq \delta T \exp( - \delta \beta^* K \Delta^2(K)).
\end{align*}
\end{theorem}

Theorem \ref{theorem::limit_cycle_metastability} (reformulated from \parencite{bressloff2020phase}) implies that the scaled process $\mathbf{x}^K/K$, when initialised in the vicinity of the stable limit cycle $\bar{\mathbf{Y}}(t)$, is expected to persist in a neighbourhood of $\bar{\mathbf{Y}}$ of order $\mathcal{O}(\Delta(K))$ for a subexponential time scale of order $\mathcal{O}(\exp(\beta^* K \Delta^2(K)))$ for some constant $\beta^* > 0$, with the caveat that the phase of the stochastic scaled process $\mathbf{x}^K(t)/K$ and the limit cycle $\mathbf{\bar{Y}}(t)$ may diverge on much earlier time scales \parencite{bressloff2020phase}.

\section{Ascent to the type 1 carrying capacity} \label{sec::ascent_carrying_capacity}

In the single type case, given reproduction decreases monotonically with population size, \textcite{jagers2011population} show that a process initialised with $o(K)$ individuals approaches the vicinity of the carrying capacity $aK$, $0 < a < 1$ with non-negligible probability within $\mathcal{O}(\log K)$ time; the logarithmic time scale of ascent is intuited from expected exponential growth in the supercritical parameter regimes attained below the carrying capacity \parencite{jagers2011population}. In our multitype setting, we introduce Assumption \ref{assumption::monotonicity} below.

\begin{assumption} \label{assumption::monotonicity}
Offspring batch sizes decrease monotonically (in the sense of strong first order stochastic dominance, Definition 4 of \parencite{kopa2018strong}) as a function of the type 1 subpopulation size, that is, for all $i \in \{1, \dots, d\}$, $c_1 \leq c_2$ and upper sets $U \subset (\mathbbm{Z}_{\geq 0})^d$ (characterised by the property that,  provided $\mathbf{w} \in U$, $\mathbf{v} \in U$ for all $\mathbf{v} \geq \mathbf{w}$), $\sum_{\mathbf{\ell} \in U} [ p_{i, \mathbf{\ell}}(c_1) - p_{i, \mathbf{\ell}}(c_2) ] \geq 0$.
\end{assumption}

Assumption \ref{assumption::monotonicity} means that the offspring batch sizes decrease monotonically with the type 1 subpopulation size in the sense of strong multivariate first order stochastic dominance \parencite{kopa2018strong}. Under Assumptions \ref{assumption::limit_theorem_scaled_carrying_capacity}, \ref{assumption::type_1_offspring}, \ref{assumption::bounded_batch_size},  \ref{assumption:density_dependent_process}, \ref{assumption::irreducibility} and \ref{assumption::monotonicity}, we can show that the type 1 subpopulation reaches the threshold $aK$, $0 < a < 1$ with non-negligible probability within $\mathcal{O}(\log(K))$ time provided $\sum^d_{i=2} x_i^K(0) > 0$; we view Theorem \ref{theorem::logarithmic_ascent} below as an analogue of Theorem 3 of \parencite{jagers2011population}. To prove this result, we first show show that the types $2, \dots, d$ subpopulations attain a size of order $\Theta(K)$ with non-negligible probability within $\mathcal{O}(\log K)$ time, whereafter the FLLN of \parencite{ethier2009markov} can be used to establish ascent of the type 1 subpopulation to any level $aK$, $0 < a < 1$ in a subsequent period of order $\mathcal{O}(1)$ time with non-negligible probability. 

We cannot establish ascent of the process $\mathbf{x}^K$ to the vicinity of the unique non-zero fixed point $K \mathbf{Y}^*$ (with $Y_1=1$) in generality in the multitype case (see the discussion of periodic orbits and asymptotic stability in Section \ref{sec::FLLN_limiting_behaviour}). However, in the special case $d=2$ with $\frac{d \lambda_\text{max}}{dx}|_{x=1} <0$ and $A'(x) \leq 0$ element-wise, we can apply Corollary \ref{corollary::FLLN_limiting_behaviour_simple}, Lemma 3.1(i) of \textcite{prodhomme2023strong} and the reasoning of Theorem  \ref{theorem::logarithmic_ascent} below to establish ascent of the process $\mathbf{x}^K$ to a neighbourhood of $K \mathbf{Y}^*$ in $\mathcal{O}(\log(K))$ time with non-negligible probability.

\begin{theorem} \label{theorem::logarithmic_ascent}
Suppose Assumptions \ref{assumption::limit_theorem_scaled_carrying_capacity}, \ref{assumption::type_1_offspring}, \ref{assumption::bounded_batch_size},  \ref{assumption:density_dependent_process}, \ref{assumption::irreducibility} and \ref{assumption::monotonicity} hold. Then given $\mathbf{x}^K(0) = \mathbf{Z}$ with $\sum^d_{i=2} Z_i > 0$, in the limit $K \to \infty$, there is a non-negligible probability with which the type 1 subpopulation size $x^K_1$ reaches the threshold $aK$, $0 < a < 1$ within a period of order $\mathcal{O}(\log(K))$.
\end{theorem}
\begin{proof}

By Assumption \ref{assumption::limit_theorem_scaled_carrying_capacity}, the Perron-Frobenius theorem (Theorem 1.5 of \parencite{seneta2006non}) and \parencite{athreya2004branching}, for $z <1$, the left eigenvector $\boldsymbol{\zeta}(z)$ corresponding to the (necessarily real) eigenvalue $\lambda_\text{max}(z) > 0$ is either element-wise positive (Case (i) of Assumption \ref{assumption::irreducibility}) or satisfies $\zeta_1(z) = 0$ and $\boldsymbol{\zeta}_{2:d}(z)>0$ (Case (ii) of Assumption \ref{assumption::irreducibility}). We normalise each vector $\boldsymbol{\zeta}(z)$ such that $\min \{\{\zeta_1(z), \dots, \zeta_d(z)\} \setminus \{0\} \} = 1$. 

For notational convenience, we set $\boldsymbol{\zeta} := \boldsymbol{\zeta}(0)$ hereafter. Under Assumption \ref{assumption::limit_theorem_scaled_carrying_capacity}, there exists $\delta \in (0, 1)$ such that $|A_{i,j}(x) - A_{i,j}(0)| < \frac{\lambda_\text{max}(0)}{2 d}$ for all $x \in [0, \delta]$. Therefore,
\begin{align}
     \sum^d_{i=1} g_i \bigg( \frac{x_1^K}{K} \bigg) x_i^K \sum^d_{j=1} \bigg[ m_{i,j} \bigg( \frac{x_1^K}{K} \bigg) - \delta_{i,j} \bigg] \zeta_j > \frac{\lambda(0)}{2} \mathbf{x}^K \boldsymbol{\zeta} \text { provided } x_1^K < \delta K. \label{eq::supercritical_bound}
\end{align}

Suppose we initialise $\mathbf{x}^K(0) = \mathbf{z}$ with $\sum^d_{i=2} z_i > 0$. Define the stopping time 
\begin{align*}
    T_\delta(K) & = \inf\{ t: \mathbf{x}^K(t) \boldsymbol{\zeta} \notin (0, \delta K) \text{ or } x_1^K(t) \geq \delta K \}.
\end{align*}

Then using Dynkin's formula (Theorem 8.2 of \parencite{jagers2016size}, re-stated in Equation (\ref{eq::dynkin_formula})), the optional stopping theorem (Theorem 3.2 of \parencite{revuz2013continuous}) and the bound (\ref{eq::supercritical_bound}),
\begin{align*}
     \EX \big[ \mathbf{x}^K \big(t \wedge T_\delta(K) \big) \boldsymbol{\zeta} \big] \geq \mathbf{z} \boldsymbol{\zeta} + \frac{\lambda_\text{max}(0)}{2} \EX \bigg[ \int^{t \wedge T_\delta(K)}_0 \mathbf{x}^K \big(s \wedge T_\delta(K) \big) \boldsymbol{\zeta} ds \bigg].
\end{align*}

Applying Fubini's theorem and Gronwall's inequality, it follows that
\begin{align}
    \EX \big[ \mathbf{x}^K \big(t \wedge T_\delta(K) \big) \boldsymbol{\zeta} \big] \geq \mathbf{z} \boldsymbol{\zeta} \EX \Big[  e^{\frac{ \lambda_\text{max}(0)}{2} (t \wedge T_\delta(K))} \Big]. \label{eq::logarithmic_ascent_intermediate_1}
\end{align}

From the construction of the stopping time $T_\delta(K)$ and provided offspring batch sizes have uniformly bounded support $[0, \Xi/d]^d$ (Assumption \ref{assumption::bounded_batch_size}), the dominated convergence theorem yields
\begin{align}
    \lim_{t \to \infty} \EX \big[ \mathbf{x}^K \big(t \wedge T_\delta(K) \big) \boldsymbol{\zeta} \big]  = \EX \big[ \mathbf{x}^K \big(T_\delta(K) \big) \boldsymbol{\zeta} \big] \leq \delta K + \Xi \lVert \boldsymbol{\zeta} \rVert. \label{eq::logarithmic_ascent_intermediate_2}
\end{align}

Using Equations (\ref{eq::logarithmic_ascent_intermediate_1}) and (\ref{eq::logarithmic_ascent_intermediate_2}), and further applying the dominated convergence theorem, it follows that
\begin{align}
    \EX \Big[ e^{\frac{ \lambda(0)}{2} T_\delta(K)} \Big] = \lim_{t \to \infty}  \EX \Big[  e^{\frac{ \lambda(0)}{2} (t \wedge T_\delta(K))} \Big] \leq \frac{\delta K +  \Xi \lVert \boldsymbol{\zeta} \rVert}{\mathbf{z} \boldsymbol{\zeta}}. \label{eq::t_delta_mgf}
\end{align}
From Equation (\ref{eq::t_delta_mgf}), we conclude that $T_\delta(K) = \mathcal{O}(\text{log}(K))$. 

To prove that $\sup_{t \in [0, T_\delta(K)]} \mathbf{x}^K(t) \boldsymbol{\zeta} \geq c K$ with non-negligible probability for some constant $c > 0$, we consider each case of Assumption \ref{assumption::irreducibility} sequentially. Hereafter, we denote by $R_i(x)$ a random vector with density $p_{(i, \boldsymbol{\ell})}(x)$. 

\begin{enumerate}
\item[Case (i)] Since $\boldsymbol{\zeta} > 0$ with $\min_i \zeta_i = 1$, it is almost surely the case that either $\mathbf{x}^K(T_\delta(K)) \boldsymbol{\zeta} = 0$ or $\mathbf{x}^K(T_\delta(K)) \boldsymbol{\zeta} \geq \delta K$. By analogy with the proof of Theorem 3 of \parencite{jagers2011population}, we thus derive an upper bound for $\mathbbm{P}(\mathbf{x}^K(T_\delta(K)) \boldsymbol{\zeta} = 0)$ by coupling the process $\mathbf{x}^K$ to the $d$-type Galton Watson branching process $\tilde{\mathbf{x}}_\delta$ with offspring distribution $R_{i}(\delta)$ for individuals of type $i \in \{1, \dots, d\}$, observing that $\tilde{\mathbf{x}}_\delta$ is both irreducible and supercritical since $\delta < 1$.

To couple $\mathbf{x}^K$ and $\tilde{\mathbf{x}}_\delta$, we begin by generating sample paths for the process $\mathbf{x}^K$, augmented with offspring batches with distribution $R_i(\delta)$. Denote by $\tau_n$ the time of the $n^\text{th}$ jump for the process $\mathbf{x}^K$, with $\tau_0 = 0$ and $\mathbf{x}^K(\tau_0) = \mathbf{z}$. Given $\mathbf{x}^K(\tau_{n-1}) = \mathbf{y}$, we sample:
\begin{itemize}
    \item An inter-jump interval
    \begin{align*}
        \tau_n - \tau_{n-1} \sim \text{Exponential} \bigg( \sum^d_{i=1} g_i(y_1/K) y_i \bigg).
    \end{align*}
    \item The type of the individual that reproduces upon the $n^\text{th}$ jump
    \begin{align*}
        s_n \sim \text{Categorical} \bigg( \frac{g_1(y_1/K) y_1}{\sum^d_{i=1} g_i(y_1/K) y_i}, \dots, \frac{g_d(y_1/K) y_d}{\sum^d_{i=1} g_i(y_1/K) y_i} \bigg).
    \end{align*}
    \item A coupled pair of random vectors $(\boldsymbol{\ell_n}, \boldsymbol{\ell'_n})$ such that
    \begin{align*}
        \boldsymbol{\ell_n} \sim R_{s_n}(y_1/K) \qquad \boldsymbol{\ell'_n} \sim R_{s_n}(\delta) \qquad \mathbbm{P}(  \boldsymbol{\ell}_n \geq \boldsymbol{\ell}_n' = 1 \, | \, z_1 \leq \delta K);
    \end{align*}
    under Assumption \ref{assumption::monotonicity}, Theorem 1 of \parencite{kopa2018strong} and Strassen's theorem \parencite{lindvall1999strassen} guarantee the existence of such a coupling.
\end{itemize}

We then set $\mathbf{x}^K(t) = \mathbf{y}$ for all $t \in [\tau_{n-1}, \tau_n)$ and $\mathbf{x}^K(\tau_n) = \mathbf{y} + \boldsymbol{\ell}_n - \mathbf{e}_{s_n}$. 

Given a realisation of $\mathbf{x}^K$, we define the stopping time
\begin{align*}
    L_\delta(K) = \inf \{ n \in \mathbbm{Z}_{\geq 0}: \mathbf{x}^K(\tau_n) \notin (0, \delta K) \} \implies \tau_{L_\delta(K)} = T_\delta(K).
\end{align*}

We then construct a thinned process $\mathbf{u}^K_\delta$ based on the first $L_\delta(K)$ jumps of $\mathbf{x}^K(\tau_n)$: upon the $n^\text{th}$ jump in the process $\mathbf{x}^K$, we remove $\boldsymbol{\ell}_n - \boldsymbol{\ell}'_n$ newly-generated offspring (chosen uniformly at random), and their descendents. By construction, under this coupling, it follows that
\begin{align*}
    \mathbbm{P} \Big( \mathbf{u}^K_\delta \big( T_\delta(K) \big) \leq \mathbf{x}^K \big( T_\delta(K) \big) \Big) = 1.
\end{align*}

In the interval $[0, T_\delta(K)]$, the death of each individual of type $i$ in the process $\mathbf{u}_\delta^K$ is associated with the birth of an i.i.d. offspring batch $R_i(\delta)$, akin to the Galton-Watson process $\tilde{\mathbf{x}}_\delta(n)$. We thus deduce that
\begin{align*}
     \mathbbm{P} \Big( \mathbf{x}^K(T_\delta(K)) = \mathbf{0} \Big)  \leq \mathbbm{P} \Big( \mathbf{u}^K_\delta \big( T_\delta(K) \big) = \mathbf{0} \Big) \leq \mathbbm{P} \Big( \lim_{n \to \infty} \tilde{\mathbf{x}}_\delta(n) = \mathbf{0} \Big),
\end{align*}
whereby
\begin{align*}
     \mathbbm{P} \Big( \mathbf{x}^K(T_\delta(K)) \boldsymbol{\zeta} \geq \delta K \Big) \geq \mathbbm{P} \Big( \lim_{n \to \infty} \lVert \tilde{\mathbf{x}}_\delta(n) \rVert = \infty \Big) > 0,
\end{align*}
where the RHS is strictly positive and independent of $K$.

\item[Case (ii)] Here, we observe that $\zeta_1=0$. We can readily adapt the argument above to conclude that
\begin{align}
     \mathbbm{P} \Big( \mathbf{x}^K \big(T_\delta(K) \big) \neq \mathbf{0} \Big) & = \mathbbm{P} \Big( \mathbf{x}^K(T_\delta(K)) \boldsymbol{\zeta} \geq \delta K \text{ or } x_1^K(T_\delta(K)) \geq \delta K \Big) \notag \\
     & \geq \mathbbm{P} \Big( \lim_{n \to \infty} \lVert \hat{\mathbf{x}}_\delta(n) \rVert = \infty \Big) > 0, \label{eq::intermediate_galton_watson_bound}
\end{align} 
where $\hat{x}_\delta$ is the $(d-1)$-type Galton Watson branching process with offspring distribution $(R_{i}(\delta))_{2:d}$ for individuals of type $i \in \{2, \dots, d\}$. 

Denote by $N_c(\tau)$ the number of busy servers at time $\tau$ in the $M^X/M/\infty$ queue with $N(0)=z_1$, arrival rate $c K g_\text{max}$, fixed batches of size $\Xi$ and service rate $g_\text{min}$ (where the constants $g_\text{min}$ and $g_\text{max}$ are given by Equations (\ref{constant::scaling_gmin}) and (\ref{constant::scaling_gmax}) respectively). Under Assumption \ref{assumption::bounded_batch_size}, we can construct a coupling such that
\begin{align*}
    N_c(\tau) \geq x_1^K(\tau) \text{  provided } \sup_{t \in [0, \tau]} \mathbf{x}^K(t) \boldsymbol{\zeta} \leq cK.
\end{align*}

By Corollary 2.2 of \parencite{daw2019distributions}, for all $\tau \geq 0$,
\begin{align*}
    \EX \big[ e^{N_c(\tau)} \big] \leq e^{z_1 + c K g_\text{max} \sum^\Xi_{j=1} {\Xi \choose j} \frac{(e-1)^j}{j g_\text{min}}}.
\end{align*}
By Markov's inequality, it then follows that
\begin{align*}
    \mathbbm{P} \bigg( & x_1^K(\tau) \geq \delta K \, \bigg| \, \sup_{t \in [0, \tau]} \mathbf{x}^K(t) \boldsymbol{\zeta} \leq cK \bigg) \leq \mathbbm{P} \big( N_c(\tau) \geq \delta K \big) \leq e^{z_1 -K \big( \delta - c g_\text{max} \sum^\Xi_{j=1} {\Xi \choose j} \frac{(e-1)^j}{j g_\text{min}} \big)}.
\end{align*}
Now, choose $c = \delta/(2 g_\text{max} \sum^\Xi_{j=1} {\Xi \choose j} \frac{(e-1)^j}{j g_\text{min}})$. Then we conclude that
\begin{align}
    \mathbbm{P} \bigg( & x_1^K \big( T_\delta(K) \big) \geq \delta K \, \text{ and } \, \sup_{t \in [0, T_\delta(K)]} \mathbf{x}^K(t) \boldsymbol{\zeta} \leq c K \bigg) \notag \\
    & \leq \mathbbm{P} \bigg( x_1^K \big( T_\delta(K) \big) \geq \delta K \, \bigg| \, \sup_{t \in [0, T_\delta(K)]} \mathbf{x}^K(t) \boldsymbol{\zeta} \leq c K \bigg) \leq e^{z_1 - \frac{\delta}{2} K} \to \infty \text{ as } K \to \infty. \label{eq::x1_hitting_bound_intermediate}
\end{align}

Using Equations (\ref{eq::intermediate_galton_watson_bound}) and (\ref{eq::x1_hitting_bound_intermediate}), it follows that $\sup_{t \in [0, T_\delta(K)]} \mathbf{x}^K(t) \boldsymbol{\zeta} \geq c K$ with non-negligible probability in the limit $K \to \infty$.
\end{enumerate}

Provided $\mathbf{x}^K(0) = \mathbf{z}$ with $\sum^d_{i=2} z_i > 0$, we have established the existence of a constant $c>0$ such that $\mathbf{x}^K \boldsymbol{\zeta}$ hits the threshold $c K$ with non-negligible probability within a time scale of order $\mathcal{O}(\log(K))$ (inferred from Equation (\ref{eq::t_delta_mgf})). In the limit $K \to \infty$, the FLLN of \parencite{ethier2009markov} (stated in Theorem \ref{theorem::FLLN}) yields strong convergence of the scaled sample paths $\mathbf{x}^K/K$ to the deterministic trajectories $\mathbf{X}$ of system of ODEs (\ref{eq::FLLN_ODE}) on finite time horizons, given $\lim_{K \to \infty} \mathbf{x}^K(0)/K = \mathbf{X}(0)$. For any $0 < a < 1$, it thus remains to bound the (deterministic) stopping time $\tau_a(\mathbf{Z}) := \inf \{ \tau \geq 0: X_1 \geq a  \, | \, \mathbf{X}(0) = \mathbf{Z}  \}$ where we can make the assumption that $\boldsymbol{\zeta}(a)^T \mathbf{Z} \geq \frac{c}{\max_{i \in \{1, \dots, d\}} \zeta_i} =: c'$.

Set $G(x) := \text{diag}(g_1(x), \dots, g_d(x))$. Under Assumption \ref{assumption::monotonicity}, the mean offspring matrix $M(x) = (m_{i,j}(x))^d_{i,j=1}$ is monotonically decreasing element-wise. Therefore, provided $X_1 \leq a$,
\begin{align*}
    \frac{d}{dt} \big(  \boldsymbol{\zeta}(a)^T \mathbf{X} \big) & = \boldsymbol{\zeta}(a)^T \big[  M(X_1)^T - I_d \big] G(X_1) \mathbf{X} \\
    & = \boldsymbol{\zeta}(a)^T A(a) G^{-1}(a) G(X_1) \mathbf{X} + \boldsymbol{\zeta}(a)^T \big[ M(X_1)^T - M(a)^T ] G(X_1) \mathbf{X} \\
    & \geq \frac{g_\text{min}}{g_\text{max}} \lambda_\text{max}(a) \boldsymbol{\zeta}^T (a) \mathbf{X}.
\end{align*}
By Gronwall's inequality, it then follows that
\begin{align}
    \boldsymbol{\zeta}(a)^T \mathbf{X}(t) \geq c' e^{\frac{g_\text{min}}{g_\text{max}} \lambda_\text{max}(a) t} \text{ for all } t < \tau_a(\mathbf{Z}). \label{eq::grownall_growth_bound_intermediate}
\end{align}

Under Assumption \ref{assumption::type_1_offspring},
\begin{align}
    \frac{dX_1}{dt} & \geq \varepsilon g_\text{min} \sum^d_{i=2} X_i - g_\text{max} a \text{ for all } t < \tau_a(\mathbf{Z}). \label{eq::type_1_growth_bound_intermediate}
\end{align}

Using Equations (\ref{eq::grownall_growth_bound_intermediate}) and (\ref{eq::type_1_growth_bound_intermediate}), it then follows that
\begin{align*}
    X_1(t) \geq c' \frac{\varepsilon g_\text{max}}{\lambda_\text{max}(a) \max_{i \in \{1, \dots, d \} } \zeta_i(a)} \Big[ e^{\frac{g_\text{min}}{g_\text{max}} \lambda(a) t} - 1 \Big] - a g_\text{max} t \text{ for all } t < \tau_a(\mathbf{Z});
\end{align*}
since $X_1(t) \leq a$ for all $0 \leq t < \tau_a(\mathbf{Z})$ by construction, we can thus conclude that $\tau_a(\mathbf{Z})$ is of order $\mathcal{O}(1)$. The statement of the proof thus follows.

\end{proof}

\section{A two-type model of immunity-modulated chronic parasitic infection} \label{sec::example_parasite_model}

We now formulate a within-host model of immunity-modulated chronic parasitic infection, in line with our motivating construction.

\subsection{On the construction of an abstracted immunity level}

Host immunity is a key determinant of within-host disease dynamics. Extended exposure to parasites is typically required to mount a robust immune response that can then curb parasite proliferation. The immune response, however, is liable to wane in the absence of exposure. Here, we adopt a heavily-simplified abstracted model of immunity akin to \parencite{mehra2024hybrid}. We assume that the death of each parasite is associated with the acquisition of one or more immunity `increments', or individuals of type 1, with a possibly state-dependent probability that is uniformly bounded below by $\varepsilon >0$ (Assumption \ref{assumption::type_1_offspring}). Each type 1 individual is modelled to possess an exponentially-distributed lifetime of mean duration $1/\varphi$. The sum of immunity increments (equivalently, the type 1 subpopulation size) then models a discretised immunity level that serves as a measure of cumulative parasite exposure on an immunologically relevant time scale \parencite{mehra2024hybrid}, with the general principle that the time scale $1/\varphi$ on which immunity increments wane is significantly longer than the expected lifespan of each parasite. We assume that no offspring are generated upon the loss of an immunity increment. The reproductive parameters of the remaining parasite types $2, \dots, d$, which may represent different parasite life stages or reservoirs (for example, tissue-stage amastigotes vs blood-stage trypomastigotes of \textit{Trypanosoma cruzi}, the causative agent of Chagas disease \parencite{de2024chagas}), are then formulated as a function of this abstracted immunity level. While the mechanisms of immune modulation are highly complex and multifaceted in reality, the notion of an abstracted immunity level has featured in various within-host models in the literature, including the discrete-time stochastic framework of \parencite{gurarie2007stochastic} in which an immunity functional governs transition rates between three malarial infection classes, and the deterministic differential equation based models of \parencite{duke2021mathematical, neant2021modeling} amongst others.

In this work, we impose a soft carrying capacity with respect to the immunity level (Assumption \ref{assumption::type_1_carrying_capacity}), to reflect a shift from supercritical parasite proliferation in an immune-naive individual to dampened subcritical parasite reproduction in an individual with substantial disease-controlling immunity. Rather than imposing an upper bound on the immunity level (as in \parencite{hartfield2015within}, where a hard carrying capacity is imposed on the immune effector population to reflect ``an intrinsic limitation in the host resources allocated to immunity''), we formulate the parasite reproductive parameters to converge in the infinite immunity limit (Assumption \ref{assumption::type_1_carrying_capacity}) whereby immunity has a saturating effect on parasite reproduction. Immunity is the key controlling mechanism under our framework. We deem host cell depletion to be negligible, a reasonable assumption if quasi-equilibrium parasitemia and host cell availability differ by multiple orders of magnitude, as in chronic Chagas disease; this differs from the classical target cell limited model of viral kinetics \parencite{perelson2002modelling}, for which quasi-equilibrium behaviour (or the ``set-point level'' in the language of viral kinetics) can be attributed to host cell depletion. Since we aim to model sustained low-level infection, we do not impose a hard carrying capacity for the pathogen burden (in contrast to  \parencite{antia1994within, kennedy2014pathogen}, where this hard carrying capacity is interpreted as a host mortality threshold). Unlike \parencite{duneau2025within}, we do not model physiological damage incurred by the host over the course of infection.

We introduce this abstracted immunity level with a view towards constructing a unified within-host framework of both acute and chronic infection. While quasi-equilibrium behaviour, compatible with sustained chronic infection, could emerge under a simpler framework with a soft carrying capacity imposed on the parasite burden itself \parencite{jagers2011population, hamza2016establishment, hognas2019lifetime} (rather than an abstracted immunity level), the resultant model would be poorly-equipped to characterise the consequences of drug treatment for chronic infection: substantial but incomplete parasite clearance could lead to a large, sudden shift in the underlying reproductive parameters of this system, pushing the model into a supercritical regime that could artificially inflate the probability of rebound (versus parasite clearance), in addition to accelerating the time scale of rebound following drug treatment. The present framework allows us to draw a distinction between acute parasite dynamics in an immune-naive individual (inoculated with a small number of parasites) versus post-treatment dynamics (driven by a comparably small number of surviving parasites) in an individual treated for chronic infection.

\subsection{Model construction}

Here, we focus on the case $d=2$ whereby we jointly model immunity (type 1) and the aggregate within-host parasite burden (type 2), without distinguishing parasite stages or reservoirs. On the state space $(\mathbbm{Z}_{\geq 0})^2$, we construct a sequence of density-dependent branching processes $\{ \mathbf{x}^K(t): t \geq 0 \}$, indexed by the type 1 carrying capacity $K$, with transition rates of the form
\begin{align*}
     \mathbf{x}^K & \to \mathbf{x}^K - (1, 0) \quad \quad \, \, \, \text{ at rate } \quad \varphi x_1^K \\
     \mathbf{x}^K & \to \mathbf{x}^K + (1, \ell - 1) \quad \text{ at rate } \quad g \bigg( \frac{x_1^K}{K} \bigg) p_\ell \bigg( \frac{x_1^K}{K} \bigg) x_2^K, \, \quad \ell \in \{0, \dots, \Xi \},
\end{align*}
where, conditional on the scaled immunity level $x_1^K/K$, we interpret $g( x_1^K/K )$ as the instantaneous death rate for each circulating parasite; and $p_\ell( x_1^K/K )$ as the probability that a parasite will generate an offspring batch of size $\ell$ (in addition to a single immunity increment) upon death. We additionally define the functions
\begin{align*}
    m(x) = \sum^\Xi_{\ell = 0} \ell p_\ell(x) \qquad v(x) =  \sum^\Xi_{\ell=0} \ell^2 p_\ell(x) - m(x)^2
\end{align*}
corresponding to the mean and variance respectively of the parasite offspring batch size as a function of the scaled immunity level. 

In accordance with Assumption \ref{assumption::limit_theorem_scaled_carrying_capacity}, we enforce $\text{sgn}(m(x) - 1) = \text{sgn}(1-x)$, whereby the process $\mathbf{x}^K$ is supercritical (subcritical) when $x_1^K<K$ ($x_1^K>K$). We further assume that $m, g: \mathbbm{R}_{\geq 0} \to \mathbbm{R}_{>0}$ are uniformly bounded and strictly positive functions that converge in the respective limits $x \to 0$ and $x \to \infty$ (whereby immunity is non-sterilising and has a saturating effect on parasite reproduction). This fulfils Assumption \ref{assumption::type_1_carrying_capacity}. Assumption \ref{assumption::type_1_offspring} is automatically satisfied, because parasite death is necessarily accompanied by the `birth' of an immunity increment, while the offspring batch size is almost surely bounded above by $\Xi$ in fulfilment of Assumption \ref{assumption::bounded_batch_size}. The irreducibility Assumption \ref{assumption::irreducibility}(ii) trivially holds since there is a single parasite type with $m(x) > 0$ for all $x$.

It is natural and biologically plausible to make the assumption that parasite offspring batch sizes decay monotonically (in the sense of first order stochastic dominance) as a function of immunity (Assumption \ref{assumption::monotonicity}), that is,
\begin{align*}
    \sum^\Xi_{\ell = c} p_\ell(x) \geq \sum^\Xi_{\ell = c} p_\ell(x') \text{ for all } c \in \{0, \dots, \Xi \} \text{ provided } x < x',
\end{align*}
which then implies that $m(x)$ is a monotonically decreasing function of $x$.

For ease of analysis, we impose several additional technical assumptions in line with Assumption \ref{assumption:density_dependent_process}. Specifically, we assume that the functions $p_\ell: \mathbbm{R}_{\geq 0} \to \mathbbm{R}$ are continuously differentiable,  with locally Lipschitz first derivative, and either strictly positive or identically zero for each $\ell \in \{0, \dots, \Xi \}$; it follows that $m, v: \mathbbm{R}_{\geq 0} \to \mathbbm{R}_{> 0}$ are also continuously differentiable. We further suppose that the function $g: \mathbbm{R}_{\geq 0} \to \mathbbm{R}_{> 0}$ is continuously differentiable with locally Lipschitz first derivative.

\subsection{Key results}

We now apply our results to characterise the behaviour of this system.

\subsubsection{Disease extinction}

Theorem \ref{theorem::extinction_almost_surely} guarantees that extinction occurs almost surely, excluding the biologically implausible prospect of unmitigated parasite replication, with the time to extinction possessing finite moments of all orders.

\subsubsection{Acute infection}

To model acute infection dynamics, we consider the setting where an immune-naive host (that is, $x_1^K(0) = 0$) is infected with a small parasite inoculum (that is, $x^K_2(0) = \mathcal{O}(1))$. We can argue heuristically that parasite dynamics (that is, the time evolution of the type 2 subpopulation) during acute infection can be approximated by a one-type supercritical linear Markovian branching process with fixed death rate $g(0)$ and fixed offspring density $p_\ell(0)$, $\ell \in \{0, \dots, \Xi \}$. We can formalise a coupling argument (following the approach of \parencite{barbour2010coupling}) by observing that the scaled immunity level upon the $r^\text{th}$ parasite reproduction event is at most $r/K$; by continuity of the immunity-dependent death rates $g(x)$ and offspring density $p_\ell(x)$, this then yields a bound of order $\mathcal{O}(r/K)$ for the transition rates of the original process $\mathbf{x}^K$ versus the approximating process upon the $r^\text{th}$ parasite reproduction event. We omit the precise argument here because supercritical linear branching processes have classically been used to model acute infection dynamics in the literature (see \parencite{capistran2018extracellular, bai2019effect}, for example).

Under our framework, acute infection may either be self-resolving (that is, die out before the immunity level approaches the carrying capacity) or progress to a secondary stage: by Theorem \ref{theorem::logarithmic_ascent}, there is a non-negligible probability with which the immunity level $x_1^K$ reaches the threshold $aK$, $0 < a < 1$, and the parasite burden $x_2^K$ attains a value of order $\Theta(K)$, within a logarithmic time scale $\mathcal{O}(\log(K))$.

\subsubsection{Scaling limits} \label{eq::parasite_model_scaling_limits}

Following an acute stage of infection, suppose that the parasite burden $x_2^K$ reaches a value of order $\Theta(K)$. Then on any finite time horizon thereafter, for sufficiently large $K$, we expect parasite counts to be well-approximated by the deterministic trajectories of the system of ODEs
\begin{align}
\begin{split}
    \frac{dX_1}{dt} &= -\varphi X_1 + g(X_1) X_2 \\
    \frac{dX_2}{dt} &= g(X_1) [ m(X_1) -1 ] X_2
\end{split} \label{eq::FLLN_parasite_model}
\end{align}
constituting the FLLN (Theorem \ref{theorem::FLLN}, based on \parencite{ethier2009markov}), albeit with a random initial condition (and a random time shift for ascent to this secondary stage of infection) \parencite{morris2024computation}. The limiting behaviour of the system (\ref{eq::FLLN_parasite_model}) is thus of interest with a view towards modelling chronic infection.

By Theorem \ref{theorem::FLLN_equilibrium}, the system (\ref{eq::FLLN_parasite_model}) possesses bounded trajectories, a saddle node at $\mathbf{0}$ and a unique non-zero equilibrium $\mathbf{Y}^* = (1, \varphi/g(1))$ in the non-negative orthant. Given $m'(1) < 0$ by assumption, it is straightforward to use the Jacobian,
\begin{align}
    \boldsymbol{\theta} = \begin{pmatrix}
        \varphi \big[ \frac{g'(1)}{g(1)} - 1 \big] &  g(1) \\  \varphi m'(1) & 0
    \end{pmatrix} \label{eq::theta_parasite_model}
\end{align}
with eigenvalues
\begin{align*}
    \lambda_\pm = \frac{\varphi}{2} \bigg( \frac{g'(1)}{g(1)} - 1 \bigg) \pm \frac{1}{2} \sqrt{\varphi^2 \bigg( \frac{g'(1)}{g(1)} - 1 \bigg)^2 + 4 \varphi g(1) m'(1)},
\end{align*}
to establish that $\mathbf{Y}^*$ is a sink if $g'(1)/g(1) < 1$; a centre if $g'(1)/g(1) = 1$ and a source if $g'(1)/g(1) > 1$. For any initial condition $\mathbf{X}_0 \in \mathbbm{R}_{\geq 0} \times \mathbbm{R}_{>0}$ with a non-zero parasite burden, we can thus apply Theorem \ref{theorem::planar_FLLN_limiting_behaviour} to conclude that the corresponding $\omega$-limit set $\omega(\mathbf{X}_0)$ is either $ \{ \mathbf{Y}^* \}$ or a periodic orbit enclosing $\mathbf{Y}^*$.

\subsubsection{Illustrative dynamics}

It may be natural to assume that the death rate $g(x)$ is a monotonically increasing function of the scaled immunity level $x$, to reflect an accelerated immunity-modulated rate of parasite clearance. However, our results suggest that this construction may give rise to self-sustained oscillations in the parasite burden. We now consider several special cases which illustrate the qualitative dynamics that may emerge under our construction.

\paragraph{A special case: $g(x)/x$ strictly decreasing} \label{eq::2D_parasite_quasiequilibrium}

If $g(x)/x$ is a strictly decreasing function of $x$ (which necessarily implies $g'(1)/g(1)<1$), we can use an identical argument to Corollary \ref{corollary::FLLN_limiting_behaviour_simple} to show that $\mathbf{Y}^*$ is an exponentially stable fixed point of (\ref{eq::FLLN_parasite_model}) with domain of attraction $\mathbbm{R}_{\geq 0} \times \mathbbm{R}_{>0}$. We thus recover quasi-equilibrium behaviour resembling that of the well-characterised single-type model with soft carrying capacity \parencite{jagers2011population, hamza2016establishment}.

Suppose an immune-naive individual (that is, $x_1^K(0) = 0$) is infected with a small parasite inoculum (that is, $x_2^K(0) = \mathcal{O}(1)$). Then we can apply the argument of Theorem \ref{theorem::logarithmic_ascent} and Lemma 3.1(i) of \textcite{prodhomme2023strong} to conclude that, with non-negligible probability, the scaled process $\mathbf{x}^K/K$ approaches the vicinity of the asymptotically stable fixed point $\mathbf{Y}^*$ within a ``transitory period'' of order $\mathcal{O}(\log(K))$, whereafter metastable behaviour is expected \parencite{barbour1976quasi, barbour2012total, prodhomme2023strong}. We interpret this as the establishment of chronic phase infection on a time scale of order $\mathcal{O}(\log(K))$, provided acute infection does not self-resolve due to demographic stochasticity in the early generations of parasite reproduction. Provided $g'(1)/g(1) \lesssim 1$, the fixed point $\mathbf{Y}^*$ is a stable spiral, whereby the the FLLN predicts dampened oscillations approaching $\mathbf{Y}^*$. In parameter regimes where the eigenvalue spectrum $\lambda_\pm$ for the Jacobian $\boldsymbol{\theta}$ (Equation (\ref{eq::theta_parasite_model})) is such that $\text{Re}(\lambda_\pm) \ll \text{Im}(\lambda_\pm)$, the work of \textcite{baxendale2011sustained} suggests that oscillations in the stochastic sample paths approaching $\mathbf{Y}^*$ may be well-approximated by ``a rotation modulated by an Ornstein-Uhlenbeck process''. 

To describe the dynamics of chronic infection, we use the CLT (Theorem \ref{theorem::central_limit}, based on \parencite{ethier2009markov}) to construct the quasi-equilibrium approximation $\mathbf{Y}^* + \mathbf{V}^*/\sqrt{K}$, where $\mathbf{V}^*$ is a stationary multivariate Ornstein-Uhlenbeck process with drift matrix $\boldsymbol{\theta}$ (Equation (\ref{eq::theta_parasite_model})) and local covariance matrix
\begin{align*}
    \boldsymbol{\Upsilon} = \varphi \begin{bmatrix}
    2 & 0 \\ 0 & v(1)
    \end{bmatrix}.
\end{align*}

Given that acute infection does not self-resolve and following the transitory period of $\mathcal{O}(\log(K))$, we use Corollary 1.3 of \textcite{prodhomme2023strong} to conclude that the scaled process $\mathbf{x}^K/K$ is well-approximated by the stationary Gaussian process $\mathbf{Y}^* + \mathbf{V}^*/\sqrt{K}$, with asymptotic precision $\alpha \log(K)/K \leq \Delta(K) \ll 1$ (for some $\alpha > 0$) on a time scale of order $\mathcal{O}(\beta K \Delta(K))$ (for some constant $\beta > 0$). We further apply Corollary 1.4 of \textcite{prodhomme2023strong} to conclude that the marginal distribution of $\sqrt{K}(\mathbf{x}^K/K - \mathbf{Y}^*)$, sampled at any time point in the chronic stage of infection, can be well-approximated for an exponential time scale $\mathcal{O}(\exp(cK))$ (for some constant $c>0$) by a multivariate normal distribution with mean $\mathbf{0}$ and covariance
\begin{align*}
    \mathbf{\Sigma} = \begin{bmatrix}
    \frac{g(1)}{g(1)-g'(1)} \big[ 1 - \frac{g(1) v(1)}{2 \varphi m'(1)} \big] & -\frac{v(1)}{2 m'(1)} \\
    -\frac{v(1)}{2 m'(1)} & \frac{g(1)}{g(1)-g'(1)} \big[ \frac{v(1)}{2} - \frac{\varphi m'(1)}{g(1)} - \frac{\varphi v(1) (g(1) - g'(1))^2}{g(1)^3 m'(1)} \big]
    \end{bmatrix}
\end{align*}
(in the sense that the Wasserstein distance, corresponding to the truncated Euclidean distance metric $d(\mathbf{y}, \mathbf{z}) = \lVert \mathbf{y} -\mathbf{z} \rVert_2 \wedge 1$, between these distributions approaches zero in the limit $K \to \infty$).

\paragraph{A special case: $g'(1)/g(1) > 1$} \label{sec::FLLN_periodic_2D_omega_limit}

When $g'(1)/g(1) > 1$, we expect the system of ODEs (\ref{eq::FLLN_parasite_model}) constituting the FLLN to exhibit periodic orbits: for any $\mathbf{X}_0 \in (\mathbbm{R}_{\geq 0} \times \mathbbm{R}_{>0}) \setminus \{ \mathbf{Y}^*\}$, the corresponding omega limit set $\omega(\mathbf{X}_0)$ is necessarily a periodic orbit enclosing the source $\mathbf{Y}^*$.

In this setting, likewise suppose that an immune naive-individual (that is, $x_1^K(0) = 0$) is infected with a small parasite inoculum (that is, $x_2^K(0) = \mathcal{O}(1)$). Then by Theorem \ref{theorem::logarithmic_ascent}, we may conclude that there is a non-negligible probability that acute infection does not self-resolve (due to demographic stochasticity in early parasite generations), and that a parasite burden of $\mathcal{O}(K)$ is attained within a period of $\mathcal{O}(\log(K))$. Thereafter, provided $K$ is sufficiently large, the FLLN of \parencite{ethier2009markov} (stated in Theorem \ref{theorem::FLLN}) yields that the time evolution of the scaled immunity level and parasite burden $\mathbf{x}^K/K$ is well-approximated by the deterministic trajectories of (\ref{eq::FLLN_parasite_model}), which are expected to eventually approach periodic orbits, on any finite time horizon.

Suppose that the deterministic system (\ref{eq::FLLN_parasite_model}) possesses an asymptotically stable limit cycle which satisfies the assumptions of Theorem \ref{theorem::limit_cycle_metastability} (based on \parencite{bressloff2020phase}). Then provided the scaled process $\mathbf{x}^K/K$ reaches the vicinity of this stable limit cycle, we expect the emergence of metastable behaviour: by Theorem \ref{theorem::limit_cycle_metastability} (based on \parencite{bressloff2020phase}), we expect the scaled stochastic sample paths $\mathbf{x}^K/K$ to remain in a neighbourhood of order $\mathcal{O}(\Delta(K))$ of this limit cycle for a sub-exponential time scale of order $\mathcal{O}(\beta^* \Delta^2(K))$ for some constant $\beta^*>0$ (although the phase for $\mathbf{x}^K/K$ compared to the deterministic trajectories of (\ref{eq::2D_limit_cycle_system}) may diverge on a much earlier time scale \parencite{bressloff2020phase}). Under such a regime, we would thus expect to see self-sustained oscillations in parasitemia over an extended period of chronic infection.

Antigenic variation, or systematic switching between parasite surface proteins with varying degrees of cross-immune memory, has been broadly recognised to generate self-sustained oscillations in parasitemia \parencite{mitchell2012synchronous}; here, we characterise an alternative oscillatory mechanism predicated on appropriate immune modulation of parasite clearance rates. Echoing the theory of biochemical oscillators, as summarised by \parencite{novak2008design}, we suggest that the interplay between a negative feedback loop (with immunity suppressing parasite reproduction) and a positive feedback loop (with parasite reproduction driving the acquisition of immunity) yields a ``dynamical hysteresis'' provided the system is ``sufficiently nonlinear''.

\paragraph{A special case: supercritical and subcritical Hopf bifurcation} \label{sec::FLLN_Hopf_2D}

We can construct natural parametric families for which limit cycles emerge through subcritical or supercritical Hopf bifurcation (Example \ref{example::2D_limit_cycle}). A numerical example of a system which can be verified to undergo a subcritical Hopf bifurcation is shown in Figure \ref{fig:2D_limit_cycle}. This example, albeit contrived, illustrates that the classical form of quasi-equilibrium behaviour characterised in Section \ref{eq::2D_parasite_quasiequilibrium} need not arise even in the case $d=2$ where the unique equilibrium point $\mathbf{Y}^*$ in the non-negative orthant is a sink, because the domain of attraction of $\mathbf{Y}^*$ does not necessarily include a neighbourhood of the origin (by Theorem \ref{theorem::logarithmic_ascent}, the scaled process $\mathbf{x}^K/K$ will reach a neighbourhood of order $\mathcal{O}(K)$ with non-negligible probability in $\mathcal{O}(\log(K))$ time, whereafter the FLLN of \parencite{ethier2009markov} yields that the scaled sample paths $\mathbf{x}^K/K$ of (\ref{eq::2D_limit_cycle_system}) on any finite time horizon for sufficiently large $K$). The functional form of the immunity-modulated parasite clearance rate $g(\cdot)$ is thus a key determinant of the dynamics that can emerge under this framework.

\section{Conclusion} \label{sec::conclusion}

In order to model the long term persistence of low-level parasitemia, we proposed a $d$-type Markovian branching process with one-type size dependence, where the controlling type (type 1) represents a discretised immunity level which governs the reproductive parameters of the parasites types $2, \dots, d$. Here, we address a slightly more general framework predicated on two principal assumptions: a soft carrying capacity $K$ imposed on the controlling type, such that reproduction is subcritical (supercritical) when the type 1 subpopulation size exceeds (falls below) $K$, with reproductive parameters converging in the infinite type 1 subpopulation size limit (Assumption \ref{assumption::type_1_carrying_capacity}); and, with probability at least $\varepsilon > 0$, the birth of at least one type 1 individual coinciding with the death of an individual of types $2, \dots, d$ (Assumption \ref{assumption::type_1_offspring}). In the special case where type 1 individuals generate no further offspring, we may interpret the type 1 subpopulation size as a measure of cumulative parasite exposure on an biologically-informed time scale and thus as an abstracted model of acquired exposure-dependent immunity. 

Two characteristics of this framework warrant tailored analysis: supercriticality on an unbounded domain $\mathbbm{Z}_{<K} \times (\mathbbm{Z}_{\geq 0})^{d-1}$ (ruling out direct generalisations of a number of results in the literature pertaining to single-type size-dependent branching process \parencite{jagers2011population, jagers2020populations} and Foster-Lyapunov conditions for the existence of quasistationary distributions \parencite{meyn1993stability, champagnat2023general}); and convergence of reproductive parameters in the infinite type 1 subpopulation size limit, rather than the infinite total population size limit (whereby results on nearly (super/sub)-critical multitype branching processes
\parencite{klebaner1989linear, klebaner1994asymptotic, gonzalez2005unlimited} are not directly applicable to our setting).

We prove that extinction occurs almost surely, with finite moments of all orders for the time to extinction, by analysing the upcrossings and downcrossings of the type 1 subpopulation size over a sufficiently large threshold $y > K$ (Theorem \ref{theorem::extinction_almost_surely}); our proof draws on ideas from \parencite{ferrari1995existence, gonzalez2005unlimited, anderson2012continuous}. We conjecture the existence of an exponential moment for the time to extinction, which is necessary for the existence of a quasistationary distribution \parencite{collet2012quasi}, but would also be sufficient in our case under the additional assumption of irreducibility on non-zero states (Remark \ref{remark::exponential_moment}, drawing on \parencite{ferrari1995existence}).
    
We then formulate a sequence of density-dependent branching processes indexed by the type 1 carrying capacity $K$ (Assumption \ref{assumption:density_dependent_process}), and analyse scaling limits using the canonical FLLN and CLT of \parencite{ethier2009markov} (Section \ref{sec::limit_theorems}), with a view towards characterising long-term persistence. Under an irreducibility condition (Assumption \ref{assumption::irreducibility}), we prove that the FLLN possesses a unique endemic equilibrium (Theorem \ref{theorem::FLLN_equilibrium}); in the planar case $d=2$, we recover a characterisation of the limiting behaviour of the FLLN based on the clasification of this fixed point (Theorem \ref{theorem::planar_FLLN_limiting_behaviour}). We show that periodic orbits may emerge even in the case $d=2$ when the unique endemic equilibrium is a sink (Example \ref{example::2D_limit_cycle}), or in the case $d=3$ with fixed death rates and offspring batch sizes that decay monotonically as a function of immunity (Example \ref{example::3D_limit_cycle}); this points to fundamentally different behaviour to single-type models with soft carrying capacity.

We consider two forms of metastable behaviour that may emerge under this framework, under the simplifying assumption that offspring batch sizes have uniformly bounded support (Assumption \ref{assumption::bounded_batch_size}). Provided the endemic equilibrium of the FLLN is a sink, and the scaled branching process hits its domain of attraction, we use the work of \parencite{prodhomme2023strong} to characterise the subexponential time scales for which the FLLN and CLT remain asymptotically valid (Section \ref{sec::metastability_sink}). Conversely, provided the FLLN possesses a stable limit cycle, and the scaled branching process approaches its vicinity, we use the work of \parencite{bressloff2020phase} to characterise the subexponential time scales for which the process persists in a neighborhood of the cycle, noting that the phase of the scaled process and the FLLN may diverge on much earlier time scales (Section \ref{sec::metastability_limit_cycle}).

Under the additional assumption that offspring batch sizes decay monotonically (in the sense of strong first order stochastic dominance \parencite{kopa2018strong}) as a function of the type 1 subpopulation size, we recover a logarithmic time scale of ascent to the type 1 carrying capacity, mirroring the single-type setting \parencite{jagers2011population}: provided the branching process is initialised with at least one individual of types $2, \dots, d$, we show that a type 1 subpopulation threshold $aK$, $0 < a < 1$ is attained with non-negligible within a logarithmic time scale (Theorem \ref{theorem::logarithmic_ascent}).

In the special case of a two-type within host model of immunity-modulated parasitic disease, we apply these results to show that the functional form of the immunity-dependent parasite clearance rate is a key determinant of qualitative dynamics, characterising the emergence of both classical quasi-equilibrium behaviour (in direct parallel to single-type models with soft carrying capacity \parencite{jagers2011population}) and possible self-sustained oscillations in different parameter regimes (Section \ref{sec::example_parasite_model}).

\section*{Declaration on the use of generative AI}

ChatGPT-5 (Oxford Edu licence, OpenAI, https://chatgpt.com/) was used to identify relevant mathematical literature and terminology, based on natural language queries or brief mathematical descriptions of sub-problems encountered over the course of this project; all relevant literature has been cited directly. ChatGPT-5 was also used to aid construction of the example in Figure \ref{fig:2D_limit_cycle}, based on a set of derived conditions (specifically, to propose a family of increasing sigmoidal functions with $g_1'''(1) - 3 g''(1) > 0$ and $g_\alpha'(1)/g_\alpha(1) =1$); and to generate Mathematica code for interactive plots that were toggled to formulate the example in Figure \ref{fig:limit_cycle_3D}.

\section*{Acknowledgements}
SM gratefully acknowledges a Keble College Sloane Robinson/Clarendon Scholarship and a Nuffield Department of Clinical Medicine Studentship from the University of Oxford.

This research was funded in part by the Wellcome Trust [315982/Z/24/Z]. For the purpose of Open Access, the author has applied a CC BY public copyright licence to any Author Accepted Manuscript version arising from this submission.

\printbibliography

\clearpage

\appendix

\counterwithin{figure}{section}

\section{Numerical examples}

\paragraph*{Subcritical and supercritical Hopf bifurcation in the case $d=2$ for a model of immunity-modulated parasitic disease}

\begin{example} \label{example::2D_limit_cycle}

Let $g_\alpha: \mathbbm{R}_{\geq 0} \to \mathbbm{R}_{>0}$ be a uniformly bounded, strictly positive and smooth family of functions such that $g_\alpha(1) = \gamma$ and $g'_\alpha(1) / g_\alpha(1) = \alpha$. Let $m:\mathbbm{R}_{\geq 0} \to \mathbbm{R}_{>0}$ be a smooth, uniformly bounded, monotonically decreasing function with $m(1)=1$. To avoid onerous calculation, we enforce $m'(1) = -\gamma/\varphi$. Consider a family of systems parametrised by $\alpha$,
\begin{align}
    \frac{d \mathbf{X}}{dt} = \begin{pmatrix}
        -\varphi & g_\alpha(X_1) \\ 0 & g_\alpha(X_1) \big[ m(X_1) - 1 \big]
    \end{pmatrix} \mathbf{X} =: A_\alpha(X_1) \mathbf{X}. \label{eq::2D_limit_cycle_system}
\end{align}

For all $\alpha$, we observe that $\mathbf{Y}^* = (1, \varphi/\gamma)$ is the unique non-zero fixed point of (\ref{eq::2D_limit_cycle_system}) in the non-negative orthant. We also note that the eigenvalues of the Jacobian $\boldsymbol{\theta}_\alpha$ evaluated at $\mathbf{Y}^*$ cross the imaginary axis transversely at $\alpha=1$, such that $\mathbf{Y}^*$ is a sink (source) provided $\alpha < 1$ ($\alpha > 1$), compatible with a Hopf bifurcation at $\alpha = 1$. 

We can reformulate this family of systems in the Poincare-Andronov-Hopf normal form (Equation 20.2.13 of \parencite{wiggins2003introduction}) by applying the transformation 
\begin{align*}
    Z_1 \mapsto X_1 - 1 \qquad Z_2 \mapsto X_2 - \frac{\varphi}{\gamma},
\end{align*}
whereby we can readily calculate the first Lyapunov exponent (Equation 20.2.14 of \parencite{wiggins2003introduction})
\begin{align}
    s := \frac{\varphi}{16 \gamma^2}   \Big( g_1''(1) \big[ \varphi m''(1) - 3 \gamma \big] + \gamma g_1'''(1) \Big). \label{eq::lyapunov_coefficient}
\end{align}

Applying Theorem 20.2.3 of \parencite{wiggins2003introduction} yields two cases:
\begin{itemize}
    \item If $s < 0$, then there is a supercritical Hopf bifurcation at $\alpha=1$ whereby there exists a constant $c_1 > 1$ such that, for $1 < \alpha < c_1$, an asymptotically stable periodic orbit appears around the unstable fixed point $\mathbf{Y}^*$. 
    \item If $s > 0$, then there is a subcritical Hopf bifurcation at $\alpha=1$ whereby there exists a constant $c_2 < 1$ such that, for $c_2 < \alpha < 1$, an unstable periodic orbit appears around the asymptotically stable fixed point $\mathbf{Y}^*$, which then necessarily possesses a bounded domain of attraction; provided $\mathbf{X}_0 \in \mathbbm{R}_{\geq 0} \times \mathbbm{R}_{>0}$ is sufficiently small (but non-zero), we would then expect the corresponding $\omega$-limit set $\omega(\mathbf{X}_0)$ to be a periodic orbit enclosing the sink $\mathbf{Y}^*$.
\end{itemize}

A numerical example of a system which undergoes a subcritical Hopf bifurcation at $\alpha = 1$ is shown in Figure \ref{fig:2D_limit_cycle}.

\begin{figure}[h!]
    \centering
    \includegraphics[width=0.85\linewidth]{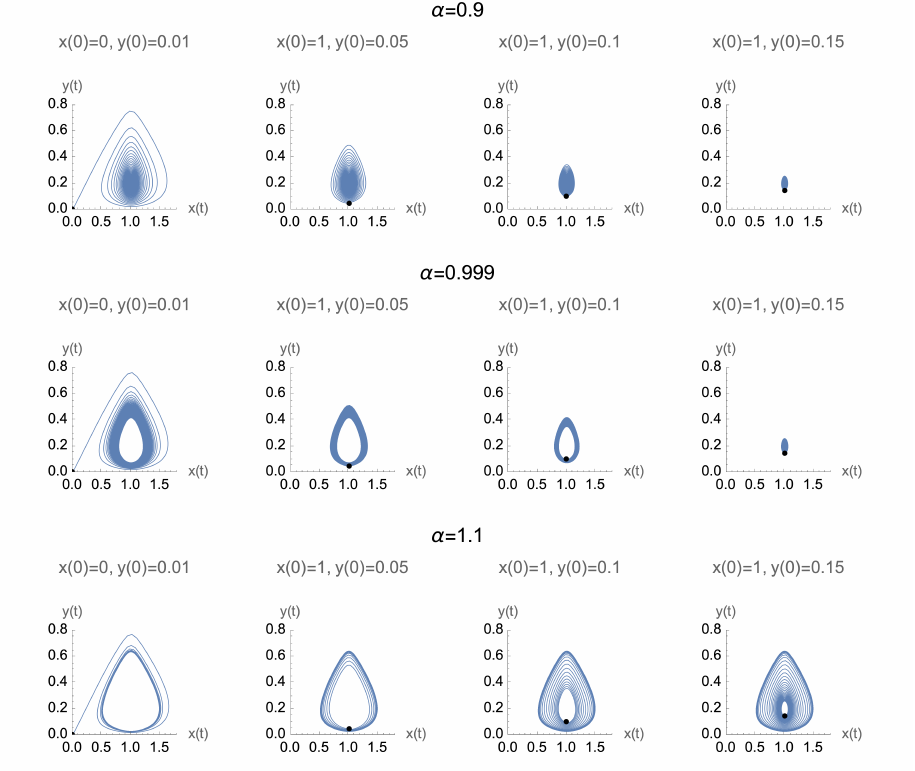}
    \caption{Trajectories of the system $\frac{d \mathbf{X}}{dt} = A_\alpha(X_1) \mathbf{X}$ (Equation (\ref{eq::2D_limit_cycle_system})) for $\alpha \in \{ 0.9, 0.999, 1.1 \}$ with\\
    \(\begin{aligned}
        \qquad \gamma = 1 \qquad \varphi = \frac{1}{5} \qquad g_\alpha(x) = \gamma \bigg( \frac{1}{2} + \frac{1}{1 + e^{-4 \alpha (x-1) + 6 \alpha^3 (x-1)^3}} \bigg) \qquad
    m(x)= \frac{1}{2} + \frac{1}{1 + e^{20 (x-1)}},
    \end{aligned}\)\\
    shown until $t=60000$, with initial conditions $\mathbf{X}(0) \in \{ (0, 0.01), (1, 0.05), (1, 0.1), (1, 0.15) \}$ (indicated with closed circles); numerically solved and plotted using Mathematica (using \texttt{NDSolve} and \texttt{ParametricPlot}). Here, we can explicitly compute the Lyapunov coefficient $s = \frac{1}{80}$ (Equation (\ref{eq::lyapunov_coefficient})) to conclude that this family of systems exhibits a subcritical Hopf bifurcation at $\alpha = 1$ around the unique non-zero fixed point $\mathbf{Y}^* = (1, \frac{1}{5})$ in the non-negative orthant.}
\label{fig:2D_limit_cycle}
\end{figure}

\end{example}

\clearpage

\paragraph*{Limit cycle in the case $d=3$, with fixed individual death rates and monotonicity in mean offspring batch sizes} \label{appendix::3D_limit_cycle}

\begin{example} \label{example::3D_limit_cycle}

Consider the family of systems
\begin{align}
    \frac{d \mathbf{X}}{dt} = \begin{pmatrix}
        -\gamma_1 & \gamma_2 & \gamma_3 \\
        0 & -\gamma_2 & \gamma_3 \mu \big[ \frac{1}{2} + \frac{1}{1 + e^{-4 \beta/\mu(X_1-1)}} \big] \\
        0 & \gamma_2 \frac{1}{\mu} \big[ \frac{1}{2} + \frac{1}{1 + e^{-4 \alpha \mu (X_1-1)}} \big]  & - \gamma_3
    \end{pmatrix} \mathbf{X} =: A_{\mu, \beta, \alpha}(X_1) \mathbf{X} \label{eq::3D_limit_cycle_example}
\end{align} 
parametrised by
\begin{align*}
    \mu = m_{3,2}(1) = \frac{1}{m_{2,3}(1)} > 0 \qquad \alpha = m_{2,3}'(1) <0  \qquad \beta = m_{3,2}'(1) <0.
\end{align*}

It is straightforward to verify that $A_{\mu, \beta, \alpha}(x)$ fulfils Assumptions \ref{assumption::limit_theorem_scaled_carrying_capacity}, \ref{assumption::type_1_offspring} and \ref{assumption::irreducibility}(ii). By Theorem \ref{theorem::FLLN_equilibrium}, the system $\frac{d \mathbf{X}}{dt} = A_{\mu, \beta, \alpha}(X_1) \mathbf{X}$ possesses a unique non-zero equilibrium with Jacobian
\begin{align*}
    \boldsymbol{\theta}_{\mu, \beta, \alpha} = \begin{pmatrix}
    -\gamma_1 & \gamma_2 & \gamma_3 \\
    \frac{\beta \gamma_1}{1 + \mu} & -\gamma_2 & \gamma_3 \mu \\
     \frac{\alpha \gamma_1 \mu}{1 + \mu} & \frac{\gamma_2}{\mu} & -\gamma_3
    \end{pmatrix}.
\end{align*}

We can verify numerically that $\boldsymbol{\theta}_{\mu, \beta, \alpha}$ is not necessarily Hurwitz stable for all $\mu > 0$, $\alpha, \beta < 0$; by Theorem 5.1 of \parencite{vassena2025mass}, this implies the existence of periodic orbits under one or more parameter regimes. A numerical example of a system which possesses an asymptotically unstable non-zero equilibrium and appears to exhibit a stable limit cycle is shown in Figure \ref{fig:limit_cycle_3D}.

\begin{figure}[h!]
    \centering
    \includegraphics[width=0.7\linewidth]{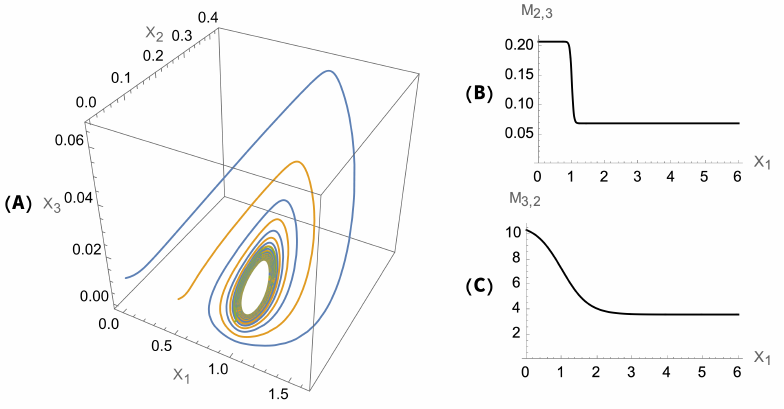}
    \caption{(A) Trajectories of the system $\frac{d \mathbf{X}}{dt} = A_{\mu, \beta, \alpha}(X_1) \mathbf{X}$ (Equation (\ref{eq::3D_limit_cycle_example})) in the case $\gamma_1 = 0.04$, $\gamma_2=\gamma_3=0.4$, $\mu=7.2$, $\alpha=-1.2, \beta = -4.8$ (whereby $\boldsymbol{\theta}_{\mu, \beta, \alpha}$ has approximate eigenvalues $-0.85354$, $0.00677 \pm 0.26408 i$), shown until $t=6000$ with initial conditions $X_2(0) = X_3(0) = 0.01$ and $X_1(0) = 0$ (blue), $X_1(0) = 0.5$ (orange) and $X_1(0) = 1$ (green); numerically solved and plotted using Mathematica (using \texttt{NDSolve} and \texttt{ParametricPlot3D}). (B,C) Mean offspring batch sizes $m_{2,3}(X_1)$ and $m_{3,2}(X_1)$.} 
    \label{fig:limit_cycle_3D}
\end{figure}
\end{example}

\end{document}